\documentclass{amsart}
\usepackage{enumerate,amssymb,amsmath,graphics,epsfig}
\usepackage[colorlinks]{hyperref}
\usepackage{color}
\usepackage{xspace}
\usepackage[labelfont=rm,font=small,labelformat=simple]{subcaption}

\def\wbox#1;#2;{\vbox{\hrule\hbox{\vrule height#1mm\kern#2mm\vrule
  height#1mm}\hrule}}

\newcommand\RE{\mathbb{R}}

\renewcommand\div{\mathop{\rm{div}}\nolimits}
\newcommand\Div{\mathop{\mathbf{div}}\nolimits}
\newcommand\Grad{\mathop{\boldsymbol\nabla}\nolimits}
\newcommand\Grads{\mathop{\boldsymbol\nabla_{\rm sym}}\nolimits}
\newcommand\grad{\mathop\nabla\nolimits}
\newcommand\Huo{H^1_0(\Omega)}
\newcommand\Ldo{L^2_0(\Omega)}

\newcommand\Oft{\Omega^f_t}
\newcommand\Ost{\Omega^s_t}
\newcommand\B{\mathcal B}
\renewcommand\u{\mathbf{u}}
\renewcommand\v{\mathbf{v}}
\renewcommand\d{\mathbf{d}}
\renewcommand\c{\mathbf{c}}
\newcommand\w{\mathbf{w}}
\newcommand\x{\mathbf{x}}
\newcommand\s{\mathbf{s}}
\newcommand\X{\mathbf{X}}
\newcommand\Y{\mathbf{Y}}
\newcommand\Z{\mathbf{Z}}
\newcommand\WW{\mathcal{W}}
\newcommand\FF{\mathbf{F}}
\newcommand\ssigma{\boldsymbol\sigma}
\newcommand\llambda{\boldsymbol\lambda}
\newcommand\mmu{\boldsymbol\mu}
\newcommand\F{\mathbb{F}}
\renewcommand\P{\mathbb{P}}

\newcommand\ds{\mathrm{d}\s}

\newcommand{\dr}{\delta\rho}
\newcommand\Hub{(H^1(\B))^d}
\newcommand\Hubd{(\Hub)'}
\newcommand\dt{\Delta t}
\newcommand\Vh{V_h}
\newcommand\Qh{Q_h}
\newcommand\Sh{S_h}
\newcommand\Lh{\Lambda_h}
\newcommand\lh{\llambda_h}

\newcommand\ucX{\u(\X(\cdot,t),t)}
\newcommand\vcX{\v(\X(t))}

\newcommand\Acca{M}
\newcommand{\ibm}{\textsc{ibm}\xspace}
\newcommand{\feibm}{\textsc{fe-ibm}\xspace}
\newcommand{\ibmdlm}{\textsc{dlm-ibm}\xspace}
\newcommand{\cfl}{\textsc{cfl}\xspace}
\newcommand{\fd}{\textsc{fd}\xspace}

\theoremstyle{plain}
\newtheorem{thm}{Theorem}
\newtheorem{proposition}[thm]{Proposition}

\newtheorem{problem}{Problem}
\newtheorem{ass}{Assumption}
\theoremstyle{remark}

\newtheorem*{remark}{Remark}

\usepackage{graphicx}
\usepackage{color,colortbl}
\usepackage{booktabs}
\definecolor{color0}{rgb}{0.4,1.0,0.4}
\definecolor{color1}{rgb}{1.0,1.0,0.4}
\definecolor{color2}{rgb}{1.0,1.0,0.4}
\definecolor{color3}{rgb}{1.0,0.501960784314,0.0}
\definecolor{color4}{rgb}{1.0,0.501960784314,0.0}
\definecolor{color5}{rgb}{1.0,0.0,0.0}
\definecolor{color6}{rgb}{1.0,0.0,0.0}
\definecolor{color7}{rgb}{1.0,0.0,0.0}
\definecolor{color8}{rgb}{1.0,0.0,0.0}
\definecolor{color9}{rgb}{1.0,0.0,0.0}

\usepackage{soul}

\let\NEW\relax

\begin{document}
\title[DLM/FD for IBM]
{The Finite Element Immersed Boundary Method with
Distributed Lagrange multiplier}
\author{Daniele Boffi}
\address{Dipartimento di Matematica ``F. Casorati'', Universit\`a di Pavia,
Italy}
\email{daniele.boffi@unipv.it}
\urladdr{http://www-dimat.unipv.it/boffi/}
\author{Nicola Cavallini}
\address{Dipartimento di Matematica ``F. Casorati'', Universit\`a di Pavia,
Italy}
\email{nicola.cavallini@unipv.it}
\urladdr{http://www-dimat.unipv.it/\~{}cavallini/}
\author{Lucia Gastaldi}
\address{DICATAM, Universit\`a di Brescia, Italy}
\email{lucia.gastaldi@unibs.it}
\urladdr{http://www.ing.unibs.it/gastaldi/}


\begin{abstract}

We introduce a new formulation for the finite element immersed boundary method
which makes use of a distributed Lagrange multiplier. We prove that a full
discretization of our model, based on a semi-implicit time advancing scheme,
is unconditionally stable with respect to the time step size.

\end{abstract}
\maketitle
\section{Introduction}
\label{se:intro}

The Immersed Boundary Method (\ibm) is an effective method for the
approximation of fluid-structure interaction problems.
After its introduction (see~\cite{PeAN} and the references therein), which was
based on a finite difference approximation of the fluid equations, several
attempts have been made in order to consider a finite element version of the
IBM; interesting results have been obtained for various formulations
(see, in
particular~\cite{feibm1,WangLiu,ZGWangLiu,BGHM3AS,BGHCAS,LiuKimTang,heltai,%
BGHP,BCG2011,costanzo}).
We shall refer to the model introduced in~\cite{feibm1} as the \feibm
formulation.

An important issue when dealing with fluid-structure interactions consists in
the choice of the time advancing scheme. For instance, it is well-known that
the Arbitrary Lagrangian Eulerian formulation (ALE), which is one of the most
popular strategies for dealing with partial differential equations defined on
moving domains, suffers from instabilities when approximating biological
models (same or similar densities for fluid and solid) unless fully implicit
schemes are used for the time evolution. This issue has been explained
in~\cite{causin} with the
help of a one dimensional model problem. On the other hand, the \feibm method
allows for a semi-implicit strategy at the price of a \cfl condition which can
be more or less severe depending of the fluid and solid dimensions
(see~\cite{BCG2011} for more details).

In the original \feibm formulation the evolution of the solid is governed by
an ordinary differential equation which reads
\[
\frac{\partial\X}{\partial t}(\s,t)=\u(\X(\s,t),t),
\]
where $\X$ denotes the position of the solid and $\u$ the fluid velocity (see
Problem~\ref{pb:pbvar} for more details).
In this paper we propose a modification to this approach with the introduction
of a suitable Lagrange multiplier. More precisely, the new equation reads
\[
\c_1(\mmu,\u(\X(\cdot,t),t))-
\c_2\Big(\mmu,\frac{\partial\X}{\partial t}(\cdot,t)\Big)=0
\qquad\forall\mmu\in\Lambda,
\]
where $\c_1(\cdot,\cdot)$ and $\c_2(\cdot,\cdot)$ are bilinear forms such that
$c_1(\mmu,\v(\X))-\c_2(\mmu,\Y)=0$ for all $\mmu\in\Lambda$ implies $\v(\X)=\Y$.

With this modification, the analogies between the \feibm and some versions of
the fictitious domain approach become more apparent. In particular, our
formulation, which is presented in Problem~\ref{pb:DLM}, can be considered as
a specific case of a fictitious domain approach with distributed Lagrange
multiplier.  For this reason we shall refer to this new formulation as the
\ibmdlm approach.
\NEW{The fictitious domain method (\fd) has been considered in several papers.
The original \fd (see, for instance,~\cite{glopanper1,glopanper2}) uses a
Lagrange multiplier along the boundary of the domain in order to deal with
boundary conditions. On the other hand, the \fd with distributed Lagrange
multiplier, presented in~\cite{girglo1995,girglopan,glokuz2007}, has been
originally introduced for the approximation of particulate flow: the fluid
domain is artificially extended to include also regions occupied by particles
(considered as solid bodies). Moreover, in~\cite{yu} the results
of~\cite{glopanhj} for the fluid/rigid-body interaction have been extended to
allow flexible bodies and a fractional step scheme is presented together with
numerical experiments. Our \ibmdlm is different from the previous
investigations, since our structure has a viscoelastic nature.}

The \ibmdlm
shares some analogies with a variational formulation presented
in~\cite{costanzo}.
One of the main differences consists in the fact that we are introducing a
Lagrange multiplier, thus giving some more flexibility to the resulting
numerical scheme.

In this paper we study our new formulation and show that it can be
successfully applied to the problem under consideration. More precisely, we
prove that a semi-implicit scheme for the time evolution of the \ibmdlm method
is unconditionally stable (no restriction at all on the time step).
Several numerical examples fully confirm our theoretical findings. \NEW{It can
be remarked that this achievement has some similarities with
the results of~\cite{newren} where an unconditionally stable discretization
for a finite difference version of the \ibm is presented.
In~\cite{newren} the key property for the proof is the
symmetry between the so-called spreading and interpolation operators and the
fact that the discretization of these two operators is performed by using the
position of the structure at the same time step. The variational
formulation, intrinsic in our finite element model, avoids the introduction of
the spreading and interpolation operators, since the Lagrangian and Eulerian
variable are linked together naturally by suitable integral terms. Moreover,
our introduction of the distributed Lagrange multiplier has the effect that
the term involving the movement of the structure is handled in a dual
form with respect to the terms containing the Lagrange multiplier (see the
symmetry of formulation~\eqref{eq:matriciona}). In our analysis, this symmetry
makes it possible to cancel out some terms (when evaluated at the same time
step). The generality of our variational formulation gives our more
flexibility in the design of the time integration scheme as well as for the
choice of the finite element spaces used for the discretization of the fluid
and of the structure. Moreover, our result holds in a more general setting
than~\cite{newren}, where the operator associated to the elastic forces is
assumed linear and self-adjoint; in our work, we suppose that the potential
energy density is convex, which is a reasonable assumption in the case of
incompressible elastic materials.

Our new formulation is also more robust than that of~\cite{costanzo}, which
does not turn out to be unconditionally stable.}

Some additional preliminary studies have been performed in~\cite{fd,ruggeri},
showing that, for some simplified formulation, the \ibmdlm approach is inf-sup
stable.

Another important issue for the approximation of incompressible fluids is
the discretization of the divergence free constraint on the fluid
velocity. This problem is related to the mass conservation properties of the
scheme and has been analyzed, in the framework of the \feibm,
in~\cite{bcgg2012,bcggumi}.
Surprisingly enough, it turns out that the \ibmdlm formulation enjoys
better mass conservation properties than the original \feibm.
For the moment, we do not have a theoretical explanation for this phenomenon,
which is clearly confirmed by our numerical experiments in
Section~\ref{se:indoeuropea}.

The structure of our paper is as follows: in Section~\ref{se:pb} we recall the
formulation of the \feibm, in Section~\ref{se:FD} we introduce our \ibmdlm
formulation, which is then approximated in time and space in
Sections~\ref{se:timescheme} and~\ref{se:fe_FD}, respectively. Finally,
Section~\ref{se:indoeuropea} reports our numerical experiments.

\section{Problem setting}
\label{se:pb}
In this section we recall the variational formulation of the \ibm
for fluid-structure interaction problems presented
in~\cite{BGHP,BCG2011}. Our formulation covers both the cases of thick and thin
structures. In the latter case the region occupied by the structure can be
represented as a domain of codimension 1 with respect to the dimension of the
fluid region. In the following subsections we introduce the problem
corresponding to the two cases.
\subsection{Thick structures}
Let $\Omega\subset\RE^d$, $d=2,3$, be a bounded domain with Lipschitz continuous
boundary. We assume that $\Omega$ is subdivided into
two connected time dependent subregions $\Oft$ and $\Ost$, which
denote the domains occupied by the fluid and the solid material, respectively.
We introduce a Lagrangian framework to deal with the motion of the
solid; hence we assume that $\Ost$ can be obtained as the image of a reference
domain $\B\subset\RE^d$. We denote by $\s$ the position of a point in $\B$ and
by $\X:\B\to\Ost$ the mapping which associates to each $\s\in\B$,
the point $\x=\X(\s,t)\in\Ost$.
The deformation gradient is defined as $\F=\Grad_s\X$, the notation
$|\F|$ indicates the determinant of the Jacobian matrix $\F$.

The constitutive equations governing the
behavior of fluids and solids are the mass balance equation and the
conservation of momenta, which in absence of external forces
can be written in strong form as follows:
\begin{equation}
\label{eq:equations}
\aligned
&\frac{d\rho}{dt}+\rho\div\u=0\\
&\rho \dot \u=\rho \frac{D\u}{dt}=
\rho\left(\frac{\partial \u}{\partial t}+\u\cdot\Grad\u\right)
=\Div\ssigma\\
\endaligned
\end{equation}
where $\u$ represents the velocity, $\ssigma$ the Cauchy stress-tensor and
$\rho$ the mass density.

We assume that both the fluid and the solid material are incompressible. This
is equivalent to impose $\div\u=0$ in $\Omega$. As a consequence, $|\F|$ is
constant in time and equals the corresponding value at time $t=0$.
In particular, $|\F|=1$ if the reference domain coincides with the initial
configuration of $\Ost$, that is $\B=\Omega^s_0$.

Assuming that the fluid and the solid material have mass densities
$\rho_f$ and $\rho_s$, respectively, with $0<\rho_f\le\rho_s$, we set
\begin{equation}
\label{eq:density}
\rho=\left\{
\begin{array}{ll}
\rho_f&\text{in }\Oft\\
\rho_s&\text{in }\Ost.
\end{array}
\right.
\end{equation}
We consider a Newtonian fluid characterized by the usual Navier--Stokes
stress tensor
\[
\ssigma_f=-p\mathbb{I}+\nu\left(\Grad\u+(\Grad\u)^\top\right),
\]
while the structure is composed by an incompressible viscous hyperelastic
material so that the Cauchy stress tensor can be separated into a fluid-like
part and an elastic part. Hence we set
\begin{equation}
\label{eq:Cauchy}
\ssigma=\left\{
\begin{array}{ll}
\ssigma_f&\text{in }\Oft\\
\ssigma_f+\ssigma_s&\text{in }\Ost.
\end{array}
\right.
\end{equation}
\NEW{
We do not make precise the constitutive relation of $\ssigma_s$, but will
provide suitable assumptions under which our theory is valid (see, in
particular, Assumption~\ref{ass}). For a more detailed discussion about
possible models we refer to~\cite{BGHP}.
}

In the following we shall deal with the quantities related to the structure
using Lagrangian variables, hence we express the elastic part of the Cauchy
stress tensor in term of the first Piola-Kirchhoff stress tensor $\P$ defined
as
\begin{equation}
\label{eq:Piola}
\P(\s,t)=|\F(\s,t)|\ssigma_s(\x,t)\F^{-\top}(\s,t)\qquad\text{for }\x=\X(\s,t).
\end{equation}
From the principle of virtual works, taking into account~\eqref{eq:density}
and~\eqref{eq:Cauchy}, we have the following problem in weak form (we refer
to~\cite{BCG2011} for a detailed derivation).
\begin{problem}
Given $\u_0\in(\Huo)^d$ and $\X_0\in W^{1,\infty}(\B)^d$,
find $(\u(t),p(t))\in(\Huo)^d\times\Ldo$ and $\X(t) \in\Hub$,
such that for almost every $t\in]0,T[$ it holds
\begin{subequations}
\begin{alignat}{2}
  &\rho_f\frac d {dt}(\u(t),\v)+b(\u(t),\u(t),\v)+a(\u(t),\v)\qquad\notag&&\\
  &\qquad-(\div\v,p(t))=\langle\d(t),\v\rangle+\langle\FF(t),\v\rangle
   &&\quad\forall\v\in(\Huo)^d
     \label{eq:NS1}\\
  &(\div\u(t),q)=0&&\quad\forall q\in\Ldo
     \label{eq:NS2}\\
  &\langle\d(t),\v\rangle=
   -\dr\int_\B\frac{\partial^2\X}{\partial t^2}\v(\X(\s,t))\,\ds
      &&\quad\forall\v\in(\Huo)^d
     \label{eq:excessvar}\\
  &\langle\FF(t),\v\rangle= - \int_\B \P(\F(\s,t)): \nabla_s
\v(\X(\s,t))\,\ds
      &&\quad\forall\v\in(\Huo)^d
     \label{eq:forzavar}\\
  &\frac{\partial\X}{\partial t}(\s,t) =\u(\X(\s,t),t) &&\quad\forall \s\in \B
     \label{eq:ode}\\
  &\u(0)=\u_0\quad\mbox{\rm in }\Omega,\qquad\X(0)=\X_0\quad\mbox{\rm in }\B.
     \label{eq:ci}
\end{alignat}
\end{subequations}
\label{pb:pbvar}
\end{problem}

Here $\dr=\rho_s-\rho_f$,
$(\cdot,\cdot)$ stands for the scalar product in $L^2(\Omega)$,
$\mathbb{D}:\mathbb{E}=
\sum_{\alpha,i=1}^d\mathbb{D}_{\alpha i}\mathbb{E}_{\alpha i}$ for all tensors
$\mathbb{D}$ and $\mathbb{E}$ 
and
\[
\aligned
a(\u,\v)&=\mmu(\Grads\u,\Grads\v)\quad \text{with }
\Grads\u=\Grad\u+(\Grad\u)^\top,\\
b(\u,\v,\w)&=
\frac{\rho_f}2\left((\u\cdot\Grad\v,\w)-(\u\cdot\Grad\w,\v)\right).
\endaligned
\]
\subsection{Thin structures}
Let us now consider the case of a thin structure with constant thickness $t_s$
very small.
Therefore we can assume that the physical quantities depend only on variables
along the middle section of the structure and are constant in the orthogonal
direction and we consider mathematically the region occupied by the structure
as a curve immersed in a two dimensional fluid or a surface in a three
dimensional one. Hence we have that $\Omega^s_t\subset\RE^d$,
$\B\subset\RE^m$ with $m=d-1$, $\X:\B\to\Omega$ and the deformation gradient
$\F:\B\to\RE^{d\times m}$ is such that
\[
|\F|=\left|\frac{\partial\X}{\partial s}\right| \text{ if }m=1,\qquad
|\F|=\left|
\frac{\partial\X}{\partial s_1}\times\frac{\partial\X}{\partial s_2}
\right|\text{ if }m=2,
\]
$s$, $s_1$ and $s_2$ being the parametric variables in $\B$.

In~\cite{BCG2011} we have shown that the variational formulation of the
fluid-structure interaction problem has the same form as Problem~\ref{pb:pbvar}
if we set $\dr=(\rho_s-\rho_f)t_s$ and $\P=t_s\tilde{\P}$ where
$\tilde{\P}$ is obtained with the due modifications from~\eqref{eq:Piola}.
Notice that $\P$ is a tensor with the same dimensions as $\F$.

\subsection{Stability estimate}
We end this section by recalling an energy estimate which follows from the
assumptions on the elastic properties of hyperelastic material which we are
considering in this paper. The results of this section can be found in more
detail in~\cite{heltai,BGHP,BCG2011}.
The hyperelastic materials are characterized by a positive energy density
$W(\F)$ which depends only on the deformation gradient. The energy density is
related to the first Piola--Kirchhoff stress tensor by means of the following
relation:
\begin{equation}
\label{eq:1st-pk}
(\P(\F(\s,t))_{\alpha i}=
\frac{\partial W}{\partial\F_{\alpha i}}(\F(\s,t))
=\left(\frac{\partial W }{\partial \F}(\F(\s,t))\right)_{\alpha i},
\end{equation}
where $i = 1,\ldots,m$ and $\alpha=1,\ldots,d$.

At the end, the elastic potential energy of the body is given by:
\begin{equation}
\label{eq:potenergy}
E\left(\X(t)\right)=\int_\B W(\F(s,t))\ds.
\end{equation}

We make the following assumption on the potential energy density $W$ which will
be useful in the following sections.

\begin{ass}
\label{ass}
We assume that $W$ is a $C^1$ convex function over the set of second order
tensors.
\end{ass}
\begin{remark}
Even though in general elasticity problems the potential energy should be
assumed polyconvex, Assumption~\ref{ass} makes sense in this context since our
material is incompressible, so that $|\F|$ is constant in time.
\end{remark}
In the following proposition, we recall the stability estimate proved
in~\cite{BCG2011}.

\begin{proposition}
\label{le:stabcont}
Let us assume that for almost every $t\in[0,T]$, $\u(t)\in(\Huo)^d$ and
$\X(t)\in(W^{1,\infty}(\B))^d$ solve
Problem~\ref{pb:pbvar}, then the following bound holds true
\begin{equation}
\label{eq:energyest}
\frac{\rho_f}2\frac{d}{dt}||\u(t)||^2_0+\nu||\Grads\u(t)||^2_0+
\frac{\dr}2\frac{d}{dt}\left\|\frac{\partial \X}{\partial t}\right\|^2_{0,\B}
+\frac{d}{dt}E(\X(t))=0,
\end{equation}
where $\|\cdot\|_0$ and $\|\cdot\|_{0,\B}$
denote the norms in $L^2(\Omega)$ and $L^2(\B)$, respectively.
\end{proposition}
The above a priori estimate states that a solution of Problem~\ref{pb:pbvar},
if exists, enjoys the following regularity properties
\[
\aligned
&\u\in C^0([0,T];(\Huo)^d)\cap H^1(0,T;(L^2(\Omega))^d)\\
&p\in L^2(0,T;H^1(\Omega))\\
&\X\in W^{1,\infty}(0,T;L^2(\B)^d)\\
& E\left(\X(t)\right)\in L^\infty(0,T).
\endaligned
\]
\section{Fictitious domain formulation with distributed Lagrange
multiplier}
\label{se:FD}
Proposition~\ref{le:stabcont} is fundamental
in the treatment of the numerical approximation of the problem; in particular,
one would like to have the same property for the discrete version of
Problem~\ref{pb:pbvar}. However, it was shown in previous papers on the finite
element discretization of the problem that the space semi-discrete version
of Problem~\ref{pb:pbvar} enjoys the same stability properties, while the
time-space discretization requires a \cfl condition which limits the time
step in terms of the size of the meshes unless a fully implicit formulation is
considered (see~\cite{BGHCAS,BGHM3AS,heltai,BCG2011}).
Here we present a new formulation of
Problem~\ref{pb:pbvar} based on the introduction of a Lagrange multiplier to
enforce the motion condition~\eqref{eq:ode} which turns out to
be unconditionally stable after time discretization.
We present separately such formulation for the case of structures of
codimension zero and one.

We consider three functional spaces
$\Lambda$, $\Acca_1$ and $\Acca_2$, and
two bilinear forms $\c_1:\Lambda\times\Acca_1\to\RE$ and
$\c_2:\Lambda\times\Acca_2\to\RE$ such that the equation
\[
\NEW{
\c_1(\mmu,\Z)-\c_2(\mmu,\Y)=0\quad\forall\mmu\in\Lambda
}
\]
implies that \NEW{$\Z=\Y$ in $\B$}.

\subsection{Thick structures}
Let us set $m=d$ and assume that for almost every $t\in]0,T[$ we have that
$\X(t)\in W^{1,\infty}(\B)^d$, so that $\u(\X(\cdot,t),t)\in\Hub$.

In this case, we set $\NEW{\Acca_1=\Hub}$, $\Acca_2=\Hub$
and we consider two possible choices for the bilinear forms
$\c_1$, $\c_2$ and the space $\Lambda$.
The most natural one is to let $\Lambda$ be the dual space of $\Hub$, that is
$\Lambda=\Hubd$, then $\c_2$ is given by the duality pairing between $\Hubd$
and $\Hub$, and $\c_1$ is defined accordingly
\begin{equation}
\label{eq:defc1}
\aligned
&\NEW{\c_1(\llambda,\Z)=\langle\llambda,\Z\rangle}\quad
&&\llambda\in\Lambda=\Hubd,\ \NEW{\Z\in\Hub}\\
&\c_2(\llambda,\Y)=\langle\llambda,\Y\rangle
&&\llambda\in\Lambda=\Hubd,\ \Y\in\Hub.
\endaligned
\end{equation}
Alternatively, we can associate to any $\llambda\in\Hubd$ an element
$\boldsymbol\varphi\in\Hub$ solution to the following variational equation:
\begin{equation}
\label{eq:isometria}
\int_{\B}(\Grad_s\boldsymbol\varphi\cdot\Grad_s\Y+\boldsymbol\varphi\cdot\Y)\,d\s=\langle\llambda,\Y\rangle,
\quad\forall\Y\in\Hub.
\end{equation}
%
%
%
We observe that the integral on the left hand side is the scalar
product in $\Hub$ and that
it is easy to show that there exists a unique $\boldsymbol\varphi\in\Hub$
satisfying~\eqref{eq:isometria} with the following relation
\[
\|\boldsymbol\varphi\|_{\Hub}=\|\llambda\|_{\Hubd}.
\]
Hence we can also define $\c_1$ and $\c_2$ as follows.
We set $\Lambda=\Hub$, and 
\begin{equation}
\label{eq:defc2}
\aligned
&\NEW{\c_1(\llambda,\Z)=\int_{\B}(\Grad_s\llambda\cdot\Grad_s\Z+\llambda\cdot\Z)\,d\s}
\quad
&&\NEW{\llambda\in\Lambda=\Hub,\ \Z\in\Hub}\\
&\c_2(\llambda,\Y)=\int_{\B}(\Grad_s\llambda\cdot\Grad_s\Y+\llambda\cdot\Y)\,d\s
&&\llambda\in\Lambda=\Hub,\ \Y\in\Hub.
\endaligned
\end{equation}
Using either definition~\eqref{eq:defc1} or~\eqref{eq:defc2},
equation~\eqref{eq:ode} can be written variationally as
\begin{equation}
\label{eq:constraint}
\c_1(\mmu,\ucX)-\c_2\left(\mmu,\frac{\partial\X}{\partial t}(t)\right)=0
\quad\forall\mmu\in \Lambda.
\end{equation}
Let us introduce a Lagrange multiplier $\llambda(t)\in\Lambda$ associated to the
constraint~\eqref{eq:constraint} and split equation~\eqref{eq:NS1} as follows:
\begin{equation}
\label{eq:split}
\aligned
&\rho_f\frac d {dt}(\u(t),\v)+b(\u(t),\u(t),\v)+a(\u(t),\v)+
\c_1(\llambda(t),\vcX)=0&&\forall\v\in(\Huo)^d\\
&\dr\left(\frac{\partial^2\X}{\partial t^2}(t),\Y\right)_{\B}+
(\P(\F(t)),\Grad_s\Y)_{\B}-\c_2(\llambda(t),\Y)=0&&\forall\Y\in\Hub,
\endaligned
\end{equation}
where we have denoted by $(\cdot,\cdot)_{\B}$ the scalar product in
$L^2(\B)^d$.
Then Problem~\ref{pb:pbvar} can be written as follows.
\begin{problem}
\label{pb:DLM}
Given $\u_0\in(\Huo)^d$ and $\X_0\in W^{1,\infty}(\B)$,
find $(\u(t),p(t))\in(\Huo)^d\times\Ldo$, $\X(t) \in\Hub$, and
$\llambda(t)\in\Lambda$,
such that for almost every $t\in]0,T[$ it holds
\begin{subequations}
\begin{alignat}{2}
  &\rho_f\frac d {dt}(\u(t),\v)+b(\u(t),\u(t),\v)+a(\u(t),\v)\qquad\notag&&\\
  &\qquad-(\div\v,p(t))+\c_1(\llambda(t),\vcX)=0
   &&\quad\forall\v\in(\Huo)^d
     \label{eq:NS1_DLM}\\
  &(\div\u(t),q)=0&&\quad\forall q\in\Ldo
     \label{eq:NS2_DLM}\\
  &\dr\left(\frac{\partial^2\X}{\partial t^2}(t),\Y\right)_{\B}+
(\P(\F(t)),\Grad_s\Y)_{\B}-\c_2(\llambda(t),\Y)=0&&\quad\forall\Y\in\Hub
     \label{eq:solid_DLM}\\
  &\c_1(\mmu,\ucX)-\c_2\left(\mmu,\frac{\partial\X}{\partial t}(t)\right)
   =0 &&\quad\forall\mmu\in \Lambda
     \label{eq:ode_DLM}\\
  &\u(0)=\u_0\quad\mbox{\rm in }\Omega,\qquad\X(0)=\X_0\quad\mbox{\rm in }\B.
     \label{eq:ci_DLM}
\end{alignat}
\end{subequations}
\end{problem}
It is easy to check that Problems~\ref{pb:pbvar} and~\ref{pb:DLM} are
equivalent if $\X(t)\in (W^{1,\infty}(\B))^d$.
Moreover, we have the following energy estimate.
\begin{proposition}
\label{pr:energy_DLM}
For almost every $t\in]0,T[$, let $\u(t)\in(\Huo)^d$ and $\X(t)\in\Hub$ be
solution of Problem~\ref{pb:DLM} with $\partial\X(t)/\partial t\in(L^2(\B))^d$
then the following energy estimate holds true
\begin{equation}
\label{eq:energy_DLM}
\frac{\rho_f}2\frac{d}{dt}||\u(t)||^2_0+\nu||\grad\u(t)||^2_0+
\frac{\dr}2\frac{d}{dt}\left\|\frac{\partial \X}{\partial t}\right\|^2_{0,\B}
+\frac{d}{dt}E(\X(t))=0.
\end{equation}
\end{proposition}
\begin{proof}
Taking $\v=\u(t)$ in~\eqref{eq:NS1_DLM} and $q=p(t)$ in~\eqref{eq:NS2_DLM}, and
recalling that $b(\u,\u,\u)=0$ by definition, we get
\[
\frac{\rho_f}2\frac d{dt}\|\u(t)\|^2_0+\nu\|\Grads\u\|^2_0+
\c_1(\llambda(t),\ucX)=0.
\]
Next we consider~\eqref{eq:solid_DLM} and take
$\Y=\partial\X(t)/\partial t$, obtaining
\[
\frac{\dr}2\frac d{dt}\left\|\frac{\partial\X}{\partial t}(t)\right\|^2_{0,\B}+
\left(\P(\F(t)),\Grad_s\left(\frac{\partial\X}{\partial t}(t)\right)\right)_{\B}
-\c_2\left(\llambda(t),\frac{\partial\X}{\partial t}(t)\right)=0.
\]
Recalling the definition of the energy density and of the elastic potential
energy~\eqref{eq:1st-pk}-\eqref{eq:potenergy}, we have
\[
\begin{split}
\left(\P(\F(t)),\Grad_s\left(\frac{\partial\X}{\partial
t}(t)\right)\right)_{\B}
&=
\int_{\B} \frac {\partial W}{\partial\FF}(\FF(\s,t))
\frac{\partial}{\partial t}\Grad_s\X(\s,t)\ds\\
&=\int_\B\frac{\partial W}{\partial\FF}(\FF(\s,t))
\frac{\partial\FF}{\partial t}(\s,t)\ds\\
&=\frac d{dt}\int_\B W(\FF(\s,t))\ds=
\frac d{dt}\left(E\left(\X(t)\right)\right).
\end{split}
\]
Combining the last equations and taking into account~\eqref{eq:ode_DLM} we
arrive at~\eqref{eq:energy_DLM}.
\end{proof}
\subsection{Thin structures}
Let us now set $m=d-1$; we observe that equation~\eqref{eq:ode} has to be
intended in the sense of traces of functions in $(\Huo)^d$. Assuming that,
for a.e. $t\in]0,T[$, $\X(t)\in(W^{1,\infty}(\B))^d$, the trace of $\u(t)$ along
$\B_t$ belongs to $(H^{1/2}(\B_t))^d$ or equivalently $\ucX\in(H^{1/2}(\B))^d$.
Hence we can set $\NEW{\Acca_1=\Acca_2=H^{1/2}(\B)^d}$, and
$\Lambda=(H^{1/2}(\B)^d)'$. Then a natural
definition for the forms $\c_1$ and $\c_2$ can be
\[
\aligned
&\NEW{\c_1(\llambda,\Z)=\langle\llambda,\Z\rangle \quad}&&
\NEW{\llambda\in\Lambda,\ \Z\in H^{1/2}(\B)^d}\\
&\c_2(\llambda,\Y)=\langle\llambda,\Y\rangle
&&\llambda\in\Lambda,\ \Y\in\NEW{H^{1/2}(\B)^d}.
\endaligned
\]
Then working as in the previous subsection we can write
Problem~\ref{pb:pbvar} in the equivalent form given by Problem~\ref{pb:DLM}
provided we use the correct bilinear forms $\c_1$ and $\c_2$.
Proposition~\ref{pr:energy_DLM} holds true without any modifications.

\section{Time advancing scheme}
\label{se:timescheme}
In this section we introduce a semi-discretization in time of
Problem~\ref{pb:DLM}. 
Let us subdivide
the time interval $[0,T]$ into $N$ equal parts and let $\dt$ be the
corresponding time step. For $n=0,\dots,N$ let $t_n=n\dt$; by the superscript
$n$ we indicate the value of an unknown function at time $t_n$.
Then a fully implicit scheme reads:
given $\u_0\in(\Huo)^d$ and $\X_0\in W^{1,\infty}(\B)$, for $n=1,\dots,N$
find $(\u^n,p^n)\in(\Huo)^d\times\Ldo$, $\X^n \in\Hub$, and
$\llambda^n\in\Lambda$, such that
\[
\aligned
  &\rho_f\left(\frac{\u^{n+1}-\u^n}{\dt},\v\right)
   +b(\u^{n+1},\u^{n+1},v)+a(\u^{n+1},\v)\\
  &\qquad\qquad\qquad
-(\div\v,p^{n+1})+\c_1(\llambda^{n+1},\NEW{\v}(\X^{n+1}))=0
   &&\forall\v\in(\Huo)^d\\
  &(\div\u^{n+1},q)=0&&\forall q\in\Ldo\\
  &\dr\left(\frac{\X^{n+1}-2\X^n+\X^{n-1}}{\dt^2},\Y\right)_{\B}+
   (\P(\F^{n+1}),\Grad_s\Y)_{\B}\\
&\qquad\qquad\qquad-\c_2(\llambda^{n+1},\Y)=0
     &&\forall\Y\in\Hub\\
  &\c_1(\mmu,\u^{n+1}(\X^{n+1}))-\c_2\left(\mmu,\frac{\X^{n+1}-\X^n}{\dt}\right)
   =0 &&\forall\mmu\in\Lambda.
\endaligned
\]
The main trouble \NEW{in view of the space discretization of} the above system
consists in the computation of the terms
involving functions belonging to $\Huo$ evaluated along the structure through
the mapping $\X^{n+1}$.
\NEW{In particular, the contributions to the discrete matrix involving the
term $\v(\X^{n+1})$ would require the evaluation of the shape functions along
$\X^{n+1}$ which has not yet been computed}.
 We then choose to evaluate these terms using the
mapping at the previous time step. Analogously, the convective term is
linearized by computing the transport velocity at the previous step.
The scheme is then modified as follows.
\begin{problem}
\label{pb:DLM_sd}
Given $\u_0\in(\Huo)^d$ and $\X_0\in W^{1,\infty}(\B)$, for $n=1,\dots,N$
find $(\u^n,p^n)\in(\Huo)^d\times\Ldo$, $\X^n \in\Hub$, and
$\llambda^n\in\Hubd$,
such that
\begin{subequations}
\begin{alignat}{2}
  &\rho_f\left(\frac{\u^{n+1}-\u^n}{\dt},\v\right)
   +b(\u^n,\u^{n+1},v)+a(\u^{n+1},\v)\notag&&\\
  &\qquad\qquad\qquad
-(\div\v,p^{n+1})+\c_1(\llambda^{n+1},v(\X^n))=0
   &&\forall\v\in(\Huo)^d
   \label{eq:NS1sd}\\
  &(\div\u^{n+1},q)=0&&\forall q\in\Ldo
   \label{eq:NS2sd}\\
  &\dr\left(\frac{\X^{n+1}-2\X^n+\X^{n-1}}{\dt^2},\Y\right)_{\B}+
   (\P(\F^{n+1}),\Grad_s\Y)_{\B}\notag&&\\
&\qquad\qquad\qquad-\c_2(\llambda^{n+1},\Y)=0
     &&\forall\Y\in\Hub
   \label{eq:solidsd}\\
  &\c_1(\mmu,\u^{n+1}(\X^n))-\c_2\left(\mmu,\frac{\X^{n+1}-\X^n}{\dt}\right)
   =0 &&\forall\mmu\in\Lambda.
\label{eq:odesd}
\end{alignat}
\end{subequations}
\end{problem}
In the previous problem, $\X^1$ can be computed, for instance, by assuming
formally $\X^{-1}=0$.

In the next proposition we prove an energy estimate similar
to~\eqref{eq:energy_DLM} when the potential
energy density $W$ is convex (see Assumption~\ref{ass}).
\begin{proposition}
\label{pr:energy_sd}
Let Assumption~\ref{ass} hold and let $\u^n\in(\Huo)^d$ and $\X^n\in\Hub$
for $n=0,\dots,N$ satisfy Problem~\ref{pb:DLM_sd} with
$\X^n\in(W^{1,\infty}(\B))^d$, then
the following estimate holds true for all $n=0,\dots,N-1$
\begin{equation}
\label{eq:energy_sd}
\aligned
&\frac{\rho_f}{2\dt}\left(\|\u^{n+1}\|^2_0-\|\u^n\|^2_0\right)+
\nu\|\Grads\u^{n+1}\|^2_0\\
&+\frac{\dr}{2\dt}\left(\left\|\frac{\X^{n+1}-\X^n}{\dt}\right\|^2_{0,\B}
-\left\|\frac{\X^n-\X^{n-1}}{\dt}\right\|^2_{0,\B}\right)+
\frac{E(\X^{n+1})-E(\X^n)}{\dt} \le0
\endaligned
\end{equation}
\end{proposition}
\begin{proof}
Working similarly as in the continuous case, we take $\v=\u^{n+1}$
in~\eqref{eq:NS1sd}, using~\eqref{eq:NS2sd} and the fact that
$b(\u^n,\u^{n+1},\u^{n+1})=0$ by definition, we have
\[
\rho_f\left(\frac{\u^{n+1}-\u^n}{\dt},\u^{n+1}\right)
+a(\u^{n+1},\u^{n+1})
+\c_1(\llambda^{n+1},\u^{n+1}(\X^n))=0.
\]
The well-known identity $2(x-y)x=x^2+(x-y)^2-y^2$ and the definition of
the bilinear form $a$ imply
\begin{equation}
\label{eq:uno}
\frac{\rho_f}{2\dt}\left(\|\u^{n+1}\|^2_0-\|\u^n\|^2_0\right)+
\nu\|\Grads\u^{n+1}\|^2_0+
\c_1(\llambda^{n+1},\u^{n+1}(\X^n))\le0.
\end{equation}
Let us now take $\Y=(\X^{n+1}-\X^n)/\dt$ in~\eqref{eq:solidsd} and
recall that $\Grad_s\X=\F$ so that
\begin{equation}
\label{eq:due}
\aligned
\frac{\dr}{\dt}&
\left(\frac{\X^{n+1}-\X^n}{\dt}-\frac{\X^n-\X^{n-1}}{\dt},
\frac{\X^{n+1}-\X^n}{\dt}\right)_{\B}\\
&+
\left(\P(\F^{n+1}),\frac{\F^{n+1}-\F^n}{\dt}\right)_{\B}
-\c_2\left(\llambda^{n+1},\frac{\X^{n+1}-\X^n}{\dt}\right)=0.
\endaligned
\end{equation}
Let us discuss in detail the second nonlinear term. We want to show its
relation with the elastic potential energy~\eqref{eq:potenergy}.
We define $\WW:[0,1]\to \RE$ as
\[
\WW(t)\mathrel{\mathop:}=W(\F^n+t(\F^{n+1}-\F^n)),
\]
hence by chain rule we have thanks to~\eqref{eq:1st-pk}
\[
\WW'(t)=\P(\F^n+t(\F^{n+1}-\F^n)):(\F^{n+1}-\F^n).
\]
Thanks to Assumption~\ref{ass}, $\WW$ is convex, therefore we have
that $\WW'(1)\ge\WW(1)-\WW(0)$ from which we obtain
\[
\aligned
\left(\P(\F^{n+1}),\frac{\F^{n+1}-\F^n}{\dt}\right)_{\B}&=
\frac1{\dt}\int_\B\WW'(1)\ds\\
&\ge\frac1{\dt}\int_{\B}(\WW(1)-\WW(0))\ds=
\frac1{\dt}(E(\X^{n+1})-E(\X^n)).
\endaligned
\]
Inserting the last inequality in~\eqref{eq:due} with standard
computations, we arrive at
\begin{equation}
\label{eq:tre}
\aligned
&\frac{\dr}{2\dt}\left(
\left\|\frac{\X^{n+1}-\X^n}{\dt}\right\|^2_{0,\B}-
\left\|\frac{\X^n-\X^{n-1}}{\dt}\right\|^2_{0,\B}\right)\\
&+\frac1{\dt}(E(\X^{n+1})-E(\X^n))-
\c_2\left(\llambda^{n+1},\frac{\X^{n+1}-\X^n}{\dt}\right)\le0.
\endaligned
\end{equation}
Summing up~\eqref{eq:uno} and~\eqref{eq:tre} and taking into
account~\eqref{eq:odesd}, we arrive at~\eqref{eq:energy_sd}.
\end{proof}
We observe that the energy estimate reported in Proposition~\ref{pr:energy_sd}
does not require any limitation on the time step.

\section{Finite element discretization}
\label{se:fe_FD}
In this section we introduce the finite element discretization of
Problem~\ref{pb:DLM_sd}. For this we
consider a family $\mathcal{T}_h$ of regular meshes in $\Omega$ and a family
$\mathcal{S}_h$ of regular meshes in $\B$. We denote by $h_x$ and $h_s$ the
meshsize of $\mathcal{T}_h$ and $\mathcal{S}_h$, respectively.
Let $\Vh\subseteq(\Huo)^d$ and $\Qh\subseteq\Ldo$ be finite
element spaces which satisfy the usual discrete ellipticity on the kernel and
the discrete inf-sup conditions for the Stokes problem~\cite{bbf}.
Moreover, we consider finite dimensional subspaces $\Sh\subseteq\Hub$ and
$\Lh\subseteq\Lambda$. Then the finite element counterpart of
Problem~\ref{pb:DLM_sd} reads.
\begin{problem}
\label{pb:DLMh}
Given $\u_{0h}\in\Vh$ and $\X_0\in W^{1,\infty}(\B)$, for $n=1,\dots,N$
find $\u_h^n,p_h^n\in\Vh\times\Qh$, $\X_h^n\in\Sh$, and
$\lh^n\in\Lh$,
such that
\begin{subequations}
\begin{alignat}{2}
  &\rho_f\left(\frac{\u_h^{n+1}-\u_h^n}{\dt},\v\right)
   +b(\u_h^n,\u_h^{n+1},v)+a(\u_h^{n+1},\v)\notag&&\\
  &\qquad\qquad\qquad
-(\div\v,p_h^{n+1})+\c_1(\lh^{n+1},v(\X_h^n))=0\qquad
   &&\forall\v\in\Vh
   \label{eq:NS1h}\\
  &(\div\u_h^{n+1},q)=0&&\forall q\in\Qh
   \label{eq:NS2h}\\
  &\dr\left(\frac{\X_h^{n+1}-2\X_h^n+\X_h^{n-1}}{\dt^2},\Y\right)_{\B}+
   (\P(\F_h^{n+1}),\Grad_s\Y)_{\B}\notag&&\\
&\qquad\qquad\qquad-\c_2(\lh^{n+1},\Y)=0
     &&\forall\Y\in\Sh
   \label{eq:solidh}\\
  &\c_1(\mmu,\u_h^{n+1}(\X_h^n))-\c_2\left(\mmu,\frac{\X_h^{n+1}-\X_h^n}{\dt}\right)
   =0 &&\forall\mmu\in\Lh.
\label{eq:odeh}
\end{alignat}
\end{subequations}
where $\F_h^{n+1}=\Grad_s\X_h^{n+1}$.
\end{problem}
\NEW{
The well-posedness of Problem~\ref{pb:DLMh} is an open problem so far. It is
clear that this is related to the validity of suitable inf-sup conditions
(see~\cite{bbf}) and that this implies some compatibility conditions among the
finite element spaces. It is to be expected that such compatibilities may be
different in the case of thin or thick structures.
Some preliminary results have been obtained in~\cite{fd} and~\cite{ruggeri}
where a simplified two-dimensional model is considered mimicking the case of a
thick structure. In those papers, all finite element spaces contain continuous
piecewise linear (triangular) or bilinear (quadrilateral) functions.
}

With the same proof of Proposition~\ref{pr:energy_sd}, it can be shown that the
solution $\u_h^n\in\Vh$ and $\X_h^n\in\Sh$ satisfies
an energy estimate analogous to~\eqref{eq:energy_sd}.
\begin{proposition}
Let Assumption~\ref{ass} hold and let $u_h^n\in\Vh$ and $\X_h^n\in\Sh$ for
$n=0,\dots,N$ satisfy Problem~\ref{pb:DLMh}. Then the following
estimate holds true for all $n=0,\dots,N-1$
\begin{equation}
\label{eq:ener_discr}
\aligned
&\frac{\rho_f}{2\dt}\left(\|\u_h^{n+1}\|^2_0-\|\u_h^n\|^2_0\right)+
\mmu\|\Grads\u_h^{n+1}\|^2_0\\
+&\frac{\dr}{2\dt}\left(\left\|\frac{\X_h^{n+1}-\X_h^n}{\dt}\right\|^2_{0,\B}
-\left\|\frac{\X_h^n-\X_h^{n-1}}{\dt}\right\|^2_{0,\B}\right)+
\frac{E(\X_h^{n+1})-E(\X_h^n)}{\dt}
\le0.
\endaligned
\end{equation}
\end{proposition}
In the next section we shall present some numerical results obtained using the
scheme given by~\eqref{eq:NS1h}--\eqref{eq:odeh} where the term related to
the Piola--Kirchhoff tensor is linear. From now on, $\P(\F)$
will be defined as follows
\begin{equation}
\P(\F)=\kappa\F=\kappa\Grad_s\X,
\label{eq:piolalin}
\end{equation}
so that Problem~\ref{pb:DLMh} can be written in matrix form as follows:
\begin{equation}
\left(
\begin{array}{cc|c|c}
\mathsf{A}&\mathsf{B}^\top&0&\mathsf{L}_f(\X_h^n)^\top\\
\mathsf{B}&0&0&0\\
\hline\\[-10pt]
0&0&\mathsf{A}_s&-\mathsf{L}^\top_s\\
\hline\\[-10pt]
\mathsf{L}_f(\X_h^n)&0&-\mathsf{L}_s&0
\end{array}
\right)
\left(
\begin{array}{c}
\u_h^{n+1}\\p_h^{n+1}\\
\hline\\[-10pt]
\X_h^{n+1} \\
\hline\\[-10pt]
\lh^{n+1}
\end{array}
\right)=
\left(
\begin{array}{c}
\mathsf{f}\\0\\
\hline\\[-10pt]
\mathsf{g}\\
\hline\\[-10pt]
\mathsf{d}
\end{array}
\right)
\label{eq:matriciona}
\end{equation}
where, denoting by $\varphi$, $\psi$, $\chi$ and $\zeta$ the basis functions 
respectively in $\Vh$, $\Qh$, $\Sh$ and $\Lh$, we have used the following
notation:
\[
\aligned
&\mathsf{A}=\frac{\rho_f}{\dt}\mathsf{M}_f+\mathsf{K}_f
\quad\text{with }(\mathsf{M}_f)_{ij}=(\varphi_j,\varphi_i),
\ (\mathsf{K}_f)_{ij}=a(\varphi_j,\varphi_i)+b(\u_h^n,\varphi_j,\varphi_i)\\
&\mathsf{B}_{ki}=-(\div\varphi_i,\psi_k)\\
&\mathsf{A}_s=\frac{\dr}{\dt^2}\mathsf{M}_s+\mathsf{K}_s
\quad\text{with }(\mathsf{M}_s)_{ij}=(\chi_j,\chi_i)_{\B},
\ (\mathsf{K}_s)_{ij}=\kappa(\Grad_s\chi_j,\Grad_s\chi_i)_{\B}\\
&(\mathsf{L}_f(\X_h^n))_{lj}=\c_1(\zeta_l,\varphi_j(\X_h^n))\\
&(\mathsf{L}_s)_{lj}=\c_2(\zeta_l,\chi_j)\\
&\mathsf{f}_i=\frac{\rho_f}{\dt}(\mathsf{M}_f\u_h^n)_i\\
&\mathsf{g}_i=\frac{\dr}{\dt^2}\left(\mathsf{M}_s(2\X_h^n-\X_h^{n-1})\right)_i\\
&\mathsf{d}_l=-\frac1{\dt}(\mathsf{L}_s\X_h^n)_l.
\endaligned
\]
\NEW{
A natural question concerning system~\eqref{eq:matriciona} is to compare its
computational cost with the one related to the solution of the \feibm scheme.
It is clear that a naive implementation of the \ibmdlm is more expensive than
\feibm. It is out of the aims of this paper to investigate efficient solvers
for~\eqref{eq:matriciona}: this will be the object of future research. In this
framework, an interesting reference is~\cite{crosetto} where a general setting
for the numerical solution of FSI schemes is presented together with a
discussion on monolithic and partitioned approaches.
}

\section{Numerical experiments}
\label{se:indoeuropea}

In this section we perform a wide set of tests that numerically explore the
scheme characteristics. 
In all simulations reported in this section the velocity and 
pressure spaces are discretized using the enhanced Bercovier--Pironneau 
element introduced in~\cite{bcgg2012}, that is 
$P_1\mathrm{iso}\, P_2/(P_1+P_0)$ element. For $k>0$, $P_k$ stands for the
space of continuous piecewise polynomials of degree not greater than $k$ and
$P_0$ is the space of piecewise constants. 
Performances of the \ibmdlm scheme will be compared
to the classical pointwise \feibm.

\NEW{The matrix form of the \ibmdlm has been presented at the end of the
previous section (see Equation~\eqref{eq:matriciona}).}

In order to make the presentation clearer, let us recall the \feibm scheme
when the Piola--Kirchhoff tensor is defined as in~\eqref{eq:piolalin}\NEW{,
which turns out to be a time-space discretization of Problem~\ref{pb:pbvar}}.

For $n=0,\dots,N-1$ we perform the following three steps.

\begin{description}
\item[Step~1] Compute the \emph{fluid-structure interaction} force vector
$\mathsf{F}^{n+1}$
as 
\[
\mathsf{F}^{n+1}_i=-\kappa(\Grad_s\X_h^n,\varphi_i(\X^n_h))_\B
\]
\item[Step~2] Solve the Navier--Stokes equations
\[
\left(
 \begin{array}{c c}
 \mathsf{A}+\frac{\delta\rho}{\Delta t}\mathsf{M}_{\mathcal B}&\mathsf{B}^\top\\
 \mathsf{B} & 0
 \end{array}\right)
\left(
\begin{array}{c}
\u_h^{n+1}\\
p_h^{n+1}
\end{array}
\right)
 =
 \left(\begin{array}{c}
\mathsf{f}+\mathsf{F}^{n+1}+\frac{\delta\rho}{\Delta t}\mathsf{M}_{\B}\u_h^n\\
 0
 \end{array}
\right).
\] 
Here $(\mathsf{M}_{\B})_{ij} = (\varphi_j(\X_h^n),\varphi_i(\X_h^n))_{\B}$.\\
\item[Step~3] Update pointwise the structure position. For $M$ 
structure points, compute for the $i$-th point position:
\begin{equation}
\label{eq:odesemi}
\frac{\X_{hi}^{n+1}-\X_{hi}^{n}}{\dt}=\u_h^{n+1}(\X_{hi}^{n})
\quad\forall i=1,\ldots,M.
\end{equation}

\end{description}

\begin{figure}[h]
 \centering
 \subcaptionbox{$t = 0.1$.}
   {\includegraphics[width=5cm]{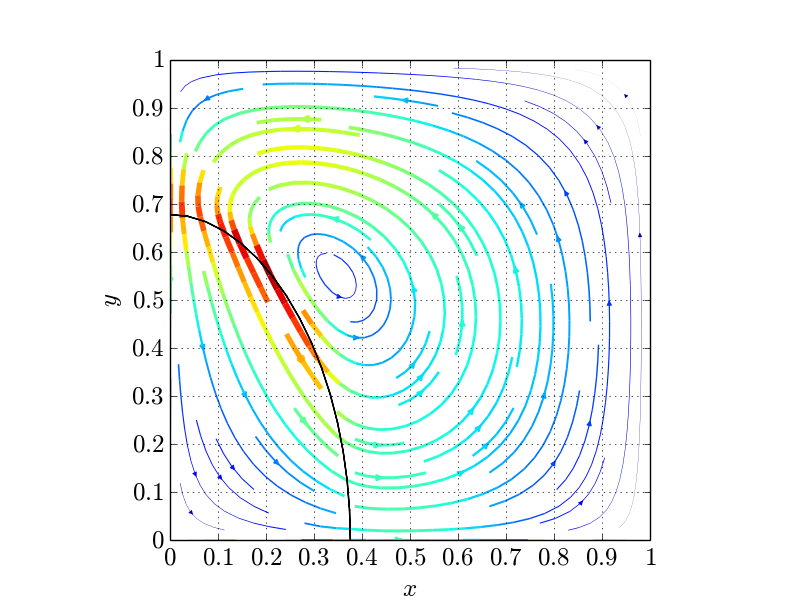}}
 \subcaptionbox{$t = 2$.}
   {\includegraphics[width=5cm]{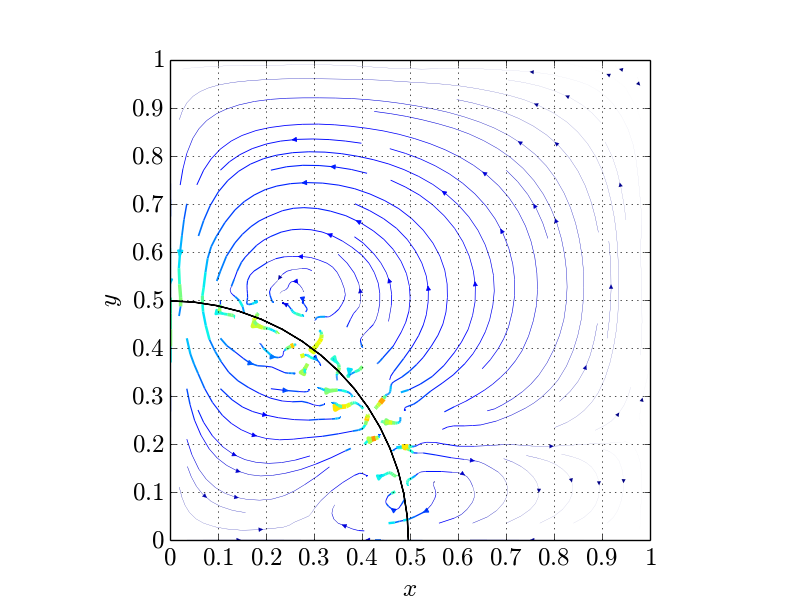}}
\caption{Codimension one structure position snapshots. The pictures represent
the velocity streamlines and the structure position for the first and final
time
steps. The streamline color pictures the velocity magnitude, red is the higher
value.}
 \label{fig:ellpse_vel_thin}
\end{figure}

\begin{figure}[h]
 \centering
 \subcaptionbox{$t = 0.1$.}
   {\includegraphics[scale=.3]{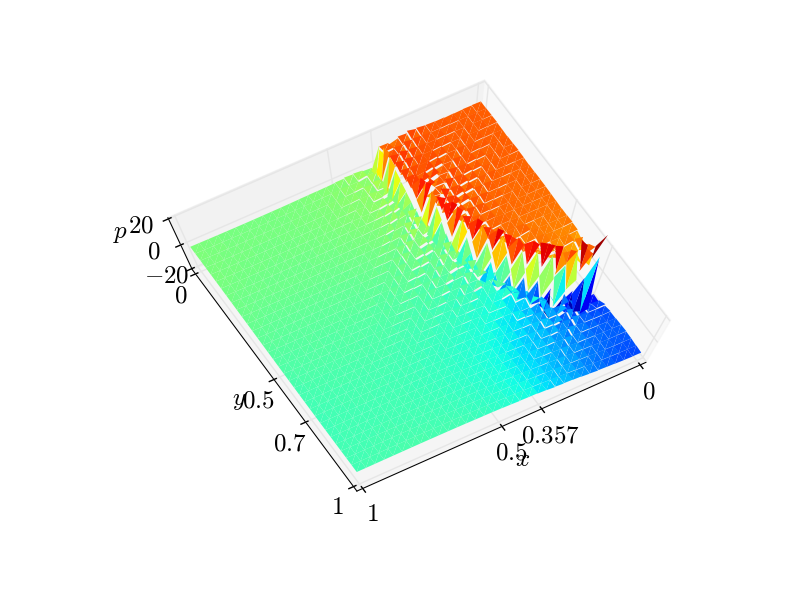}}
 \subcaptionbox{$t = 2$.}
   {\includegraphics[scale=.3]{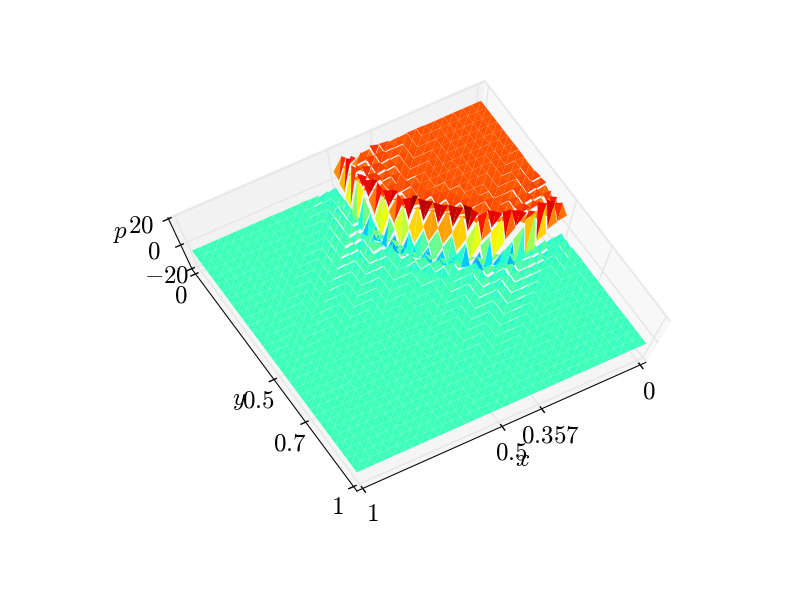}}
 \caption{Codimension one pressure snapshots. The pictures represent the 
 resulting pressure map for the first and final time 
 steps.}
 \label{fig:ellpse_prex_thin}
 \end{figure}

The first goal of our numerical experiments is to confirm the better
behavior of \ibmdlm with respect to \feibm for what the \cfl condition is
concerned. To this aim, we start recalling the \cfl condition that has been
proved in~\cite{BCG2011}. In two space dimensions, when the solid has
codimension one, the time step $\dt$ should be smaller than a multiple of
$h_x h_s$, while when the solid has codimension zero, then $\dt$ has to be
bounded by a multiple of $h_x$. On the other hand, in
Proposition~\ref{pr:energy_sd} we proved that the \ibmdlm is unconditionally
stable, that is no restriction on $\dt$ is required for its stability.
 
 \begin{figure}
 \centering
 \subcaptionbox{$t = 0.1$.}
    {\includegraphics[scale=.3]{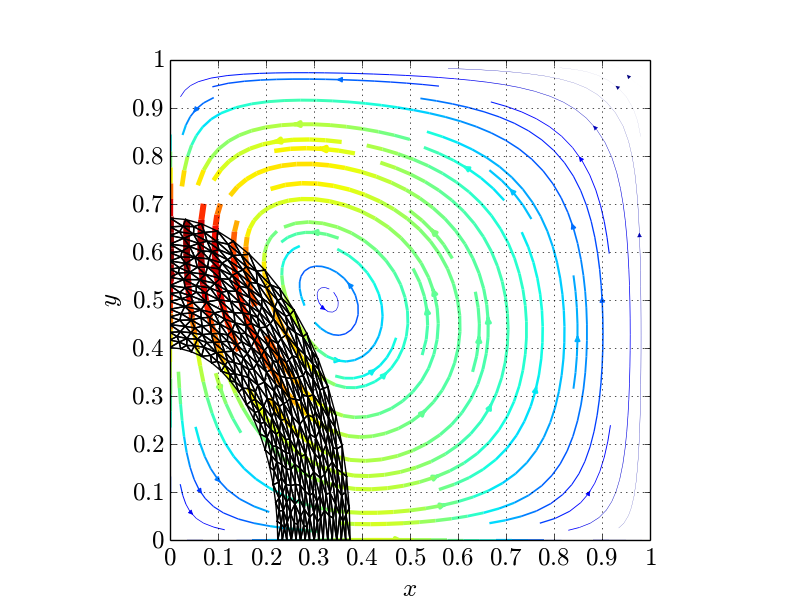}}
 \subcaptionbox{$t = 2$.}
   {\includegraphics[scale=.3]{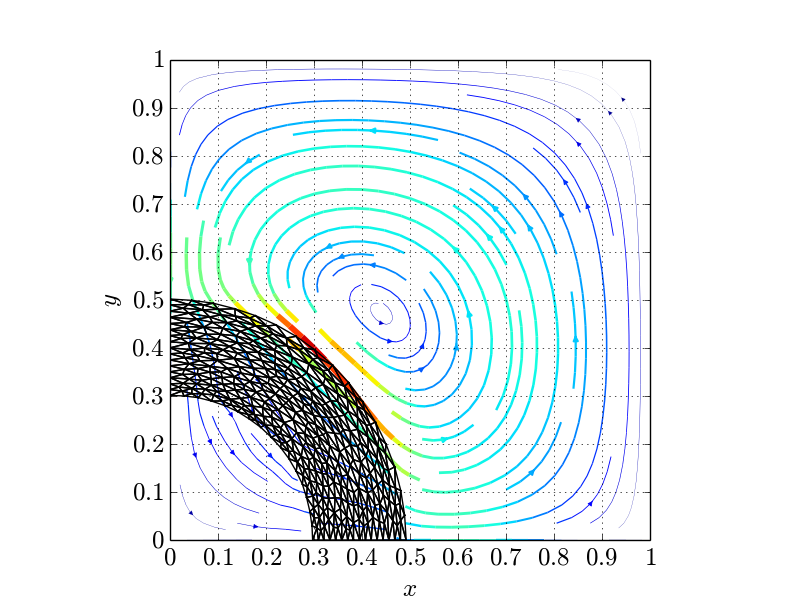}}
 \caption{Codimension zero structure position snapshots. The pictures represent the 
 velocity streamlines and the structure position for the first and final time 
 steps. The streamline color pictures the velocity magnitude, red is the higher value.}
 \label{fig:ellpse_vel_thick}
\end{figure}
 
\begin{figure}
 \centering
 \subcaptionbox{$t = 0.1$.}
   {\includegraphics[scale=.3]{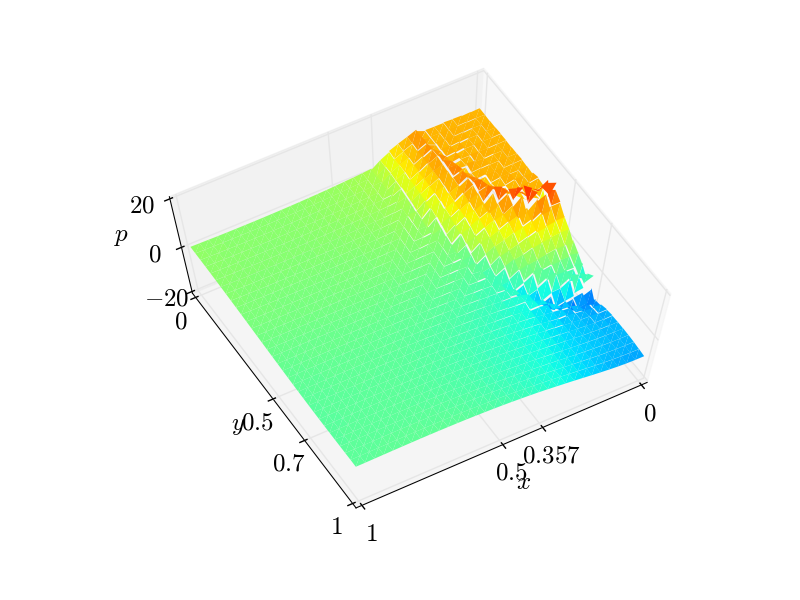}}
 \subcaptionbox{$t = 2$.}
   {\includegraphics[scale=.3]{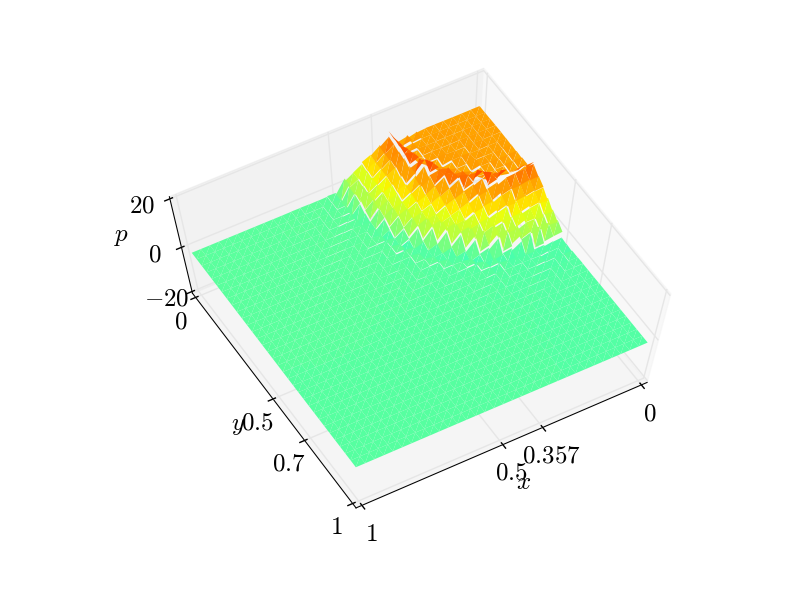}}
 \caption{Codimension zero pressure snapshots. The pictures represent the 
 resulting pressure map for the first and final time 
 steps.}
 \label{fig:ellpse_prex_thick}
 \end{figure}

A series of tests will be performed, using the classical benchmark problem of
an ellipsoidal structure that evolves to a circular equilibrium position.
The ellipsoid is centered at the midpoint of our (square) physical domain
and the initial fluid is at rest, so that we can reduce the computational
domain to a quarter of the physical one by symmetry conditions.
Figure~\ref{fig:ellpse_vel_thin} reports the initial and final configurations
of our test (codimension one) as a result of a computation performed with
\ibmdlm. The corresponding pressure is plotted in
Figure~~\ref{fig:ellpse_prex_thin}. Analogue plots are reported in
Figures~\ref{fig:ellpse_vel_thick} and~\ref{fig:ellpse_prex_thick},
respectively, in the case of codimension zero structure.
The $x$ and $y$ thicks on the pressure plot correspond to the initial and
final position of the structure.
 
In order to check the stability with respect to the time step, we consider
the fluid and structure kinematic and elastic energy:
\begin{equation}
\label{eq:ener_comp}
\Pi(\X_h^{n},\u_h^{n}) = \frac{\rho_f}{2}\|\u_h^{n}\|^2_0
+\frac{\dr}{2}\left\|\frac{\X_h^{n}-\X_h^{n-1}}{\dt}\right\|^2_{0,\B}
+
E(\X_h^{n}).
\end{equation}
We compute the energy ratio: $\Pi(\X_h^{n},\u_h^{n})/\Pi(\X_h^{0},\u_h^{0})$
for different parameter definitions as a function of time for both \feibm and
\ibmdlm computations.
Figures~\ref{fig:thin_energy_rho00} and~\ref{fig:thin_energy_rho03} show the
results in the case of the codimension one structure when $\delta\rho = 0$ and
$\delta\rho = 0.3$, respectively.
Figure~\ref{fig:thick_energy_rho03} reports on the case of codimension zero
structure when $\delta\rho = 0.3$. 

More precisely, Figures~\ref{fig:thin_energy_rho00}
and~\ref{fig:thin_energy_rho03} show the energy ratio for
a fixed $h_x=1/32$, and varying $h_s$ (row-wise) and $\dt$ (column-wise). It
is clear that the \feibm (dashed curve) has a blowing-up energy if $h_s$ is too
small compared to $\dt$.

Figure~\ref{fig:thick_energy_rho03} shows the energy ratio
for a fixed $h_s=1/8$, and varying $h_x$ (row-wise) and $\dt$ (column-wise).
It can be seen that the energy computed with the \feibm (dashed curve) blows
up if $\dt$ is not small enough with respect to $h_x$.

On the other hand, in all cases it can be appreciated that the \ibmdlm does
not need any constraint on $\dt$ in order to be stable,

 \begin{figure}
 \centering
 \subcaptionbox{\tiny{$\Delta t = 10^{-1}$, $h_s = 1/8$}.}
   {\includegraphics[scale=.2]{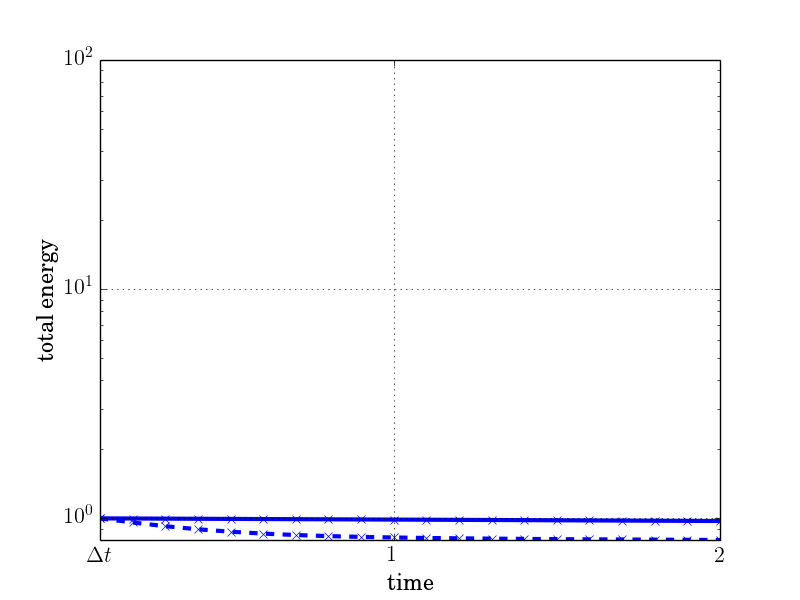}}
 \subcaptionbox{\tiny{$\Delta t = 10^{-1}$, $h_s = 1/16$}.}
   {\includegraphics[scale=.2]{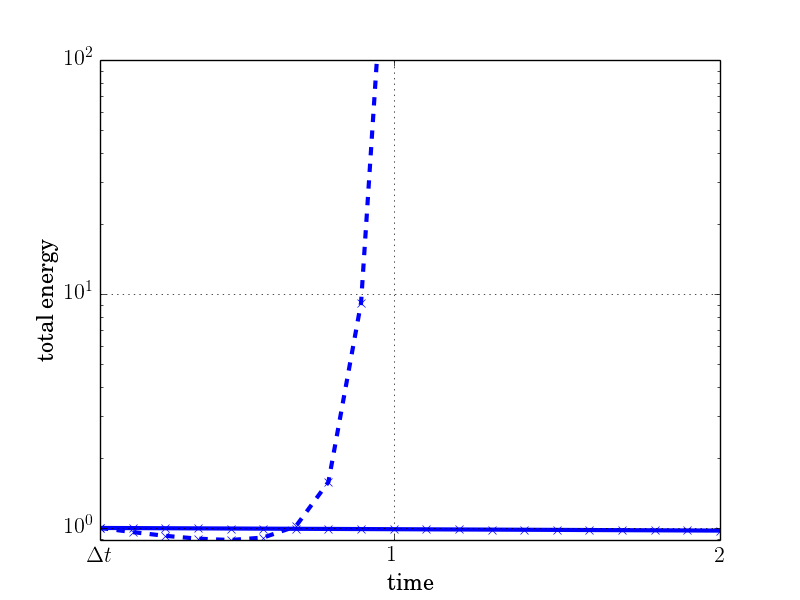}\label{subfig:cfl00}}
 \subcaptionbox{\tiny{$\Delta t = 10^{-1}$, $h_s = 1/32$}.}
   {\includegraphics[scale=.2]{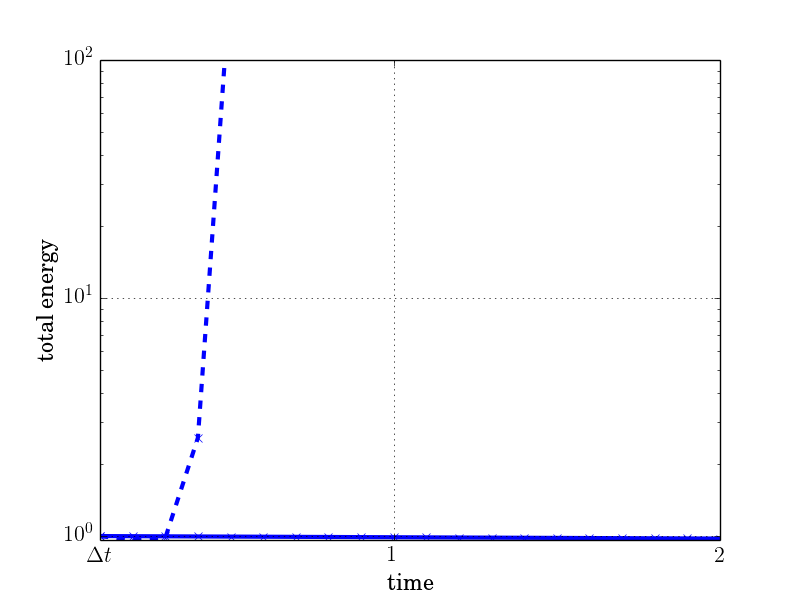}\label{subfig:cfl01}}\\
 \subcaptionbox{\tiny{$\Delta t = 5\cdot 10^{-2}$, $h_s = 1/8$}.}
   {\includegraphics[scale=.2]{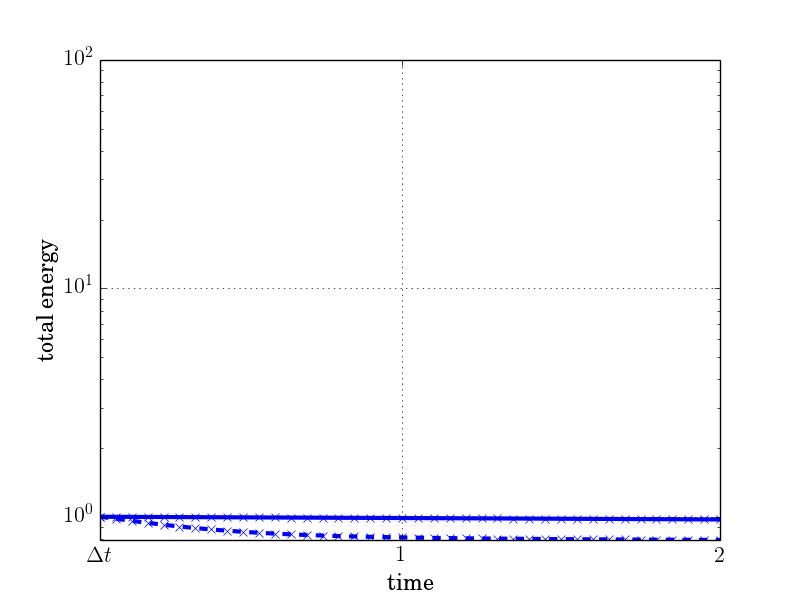}}
 \subcaptionbox{\tiny{$\Delta t = 5\cdot 10^{-2}$, $h_s = 1/16$}.}
   {\includegraphics[scale=.2]{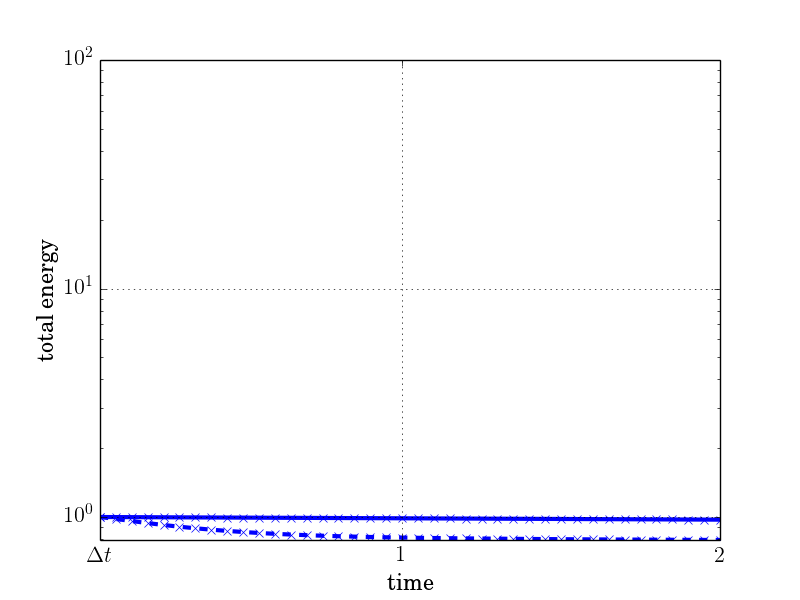}}
 \subcaptionbox{\tiny{$\Delta t = 5\cdot 10^{-2}$, $h_s = 1/32$}.}
   {\includegraphics[scale=.2]{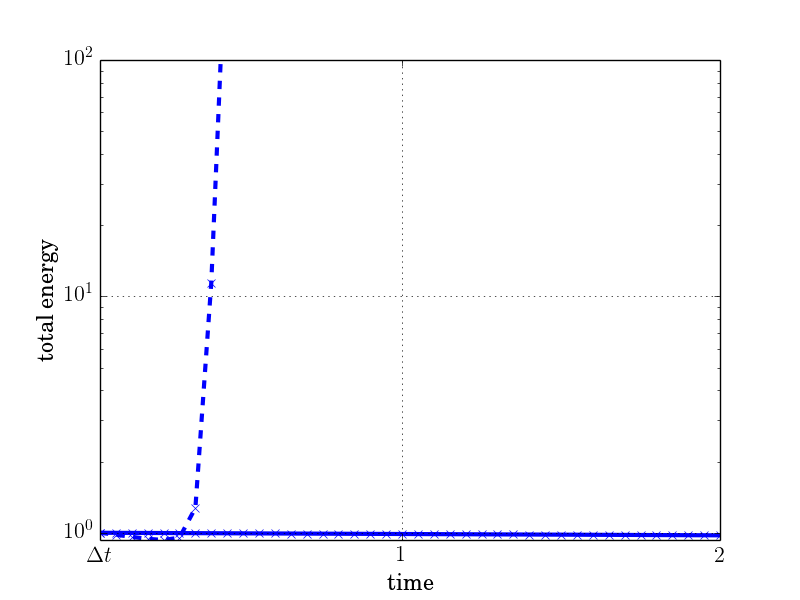}\label{subfig:cfl02}}
 \caption{Energy ratio \NEW{(see Equation~\eqref{eq:ener_comp})} for codimension one structure. The structure elastic constant $\kappa = 5$, $h_x = 1/32$,
 the fluid viscosity $\nu = 1$, $\delta\rho = 0$. The solid line correspond to the \ibmdlm scheme, while 
 the dashed line marks the energy for the \feibm scheme.
 }
 \label{fig:thin_energy_rho00}
 \end{figure}

 \begin{figure}
 \centering
 \subcaptionbox{\tiny{$\Delta t = 10^{-1}$, $h_s = 1/8$}.}
   {\includegraphics[scale=.2]{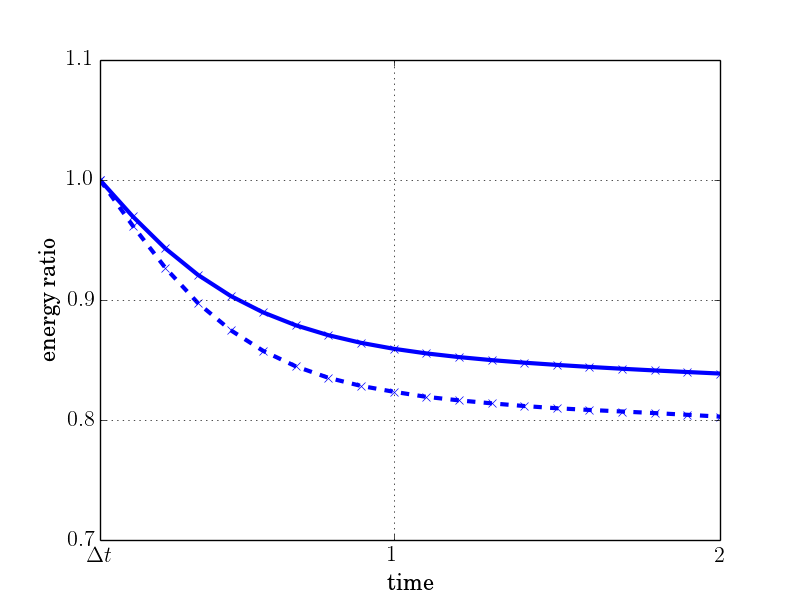}}
 \subcaptionbox{\tiny{$\Delta t = 10^{-1}$, $h_s = 1/16$}.}
   {\includegraphics[scale=.2]{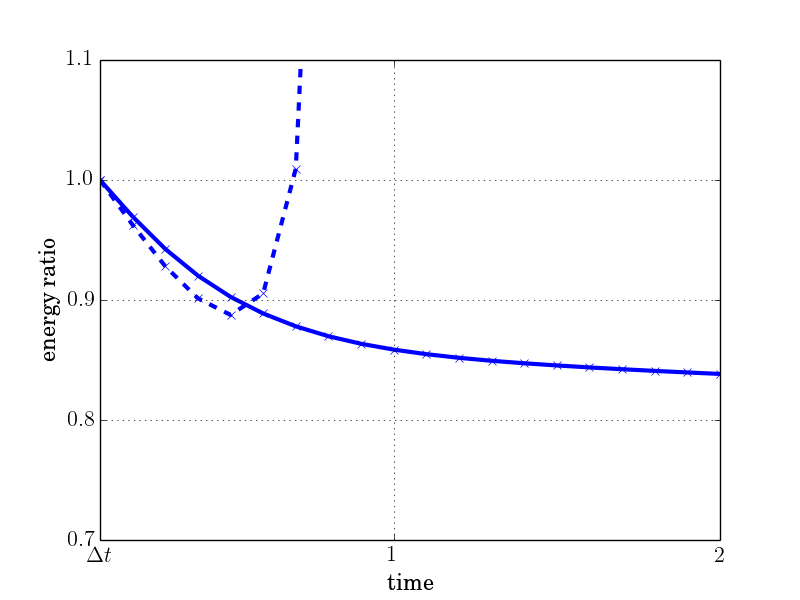}}
 \subcaptionbox{\tiny{$\Delta t = 10^{-1}$, $h_s = 1/32$}.}
   {\includegraphics[scale=.2]{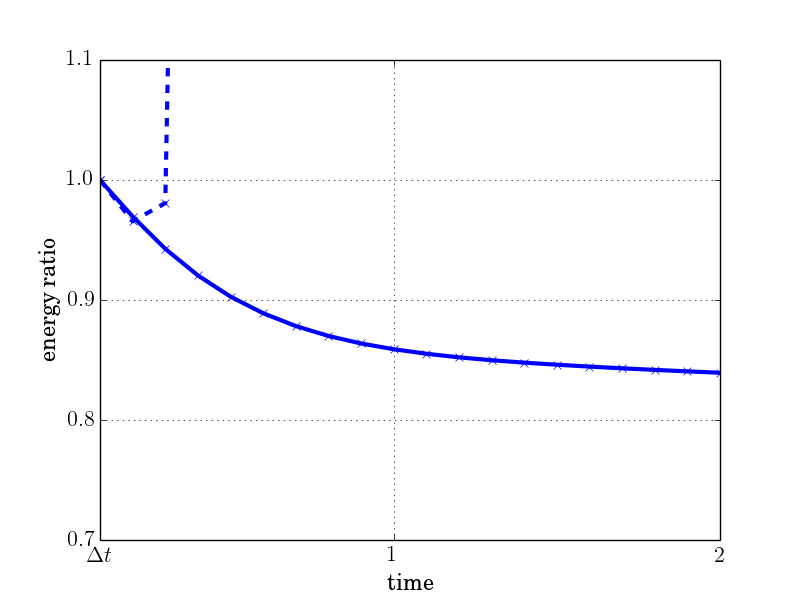}}\\
 \subcaptionbox{\tiny{$\Delta t = 5\cdot 10^{-2}$, $h_s = 1/8$}.}
   {\includegraphics[scale=.2]{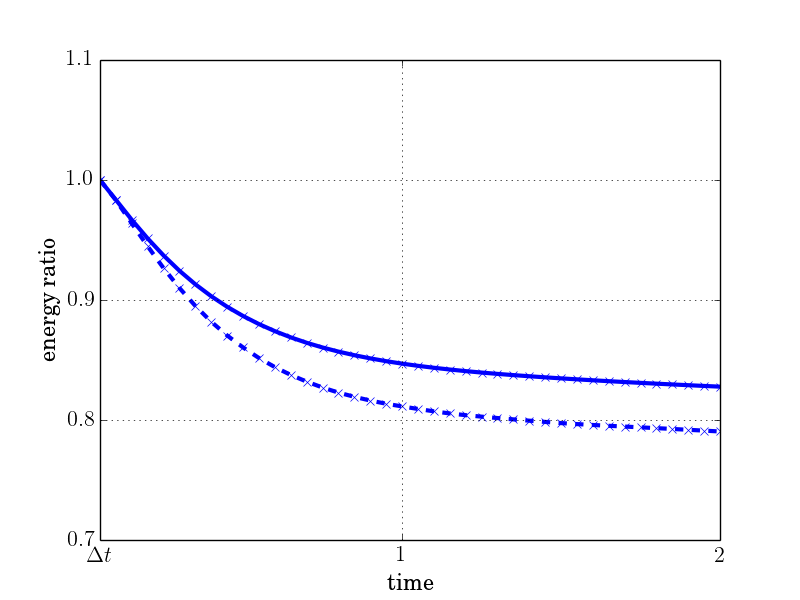}}
 \subcaptionbox{\tiny{$\Delta t = 5\cdot 10^{-2}$, $h_s = 1/16$}.}
   {\includegraphics[scale=.2]{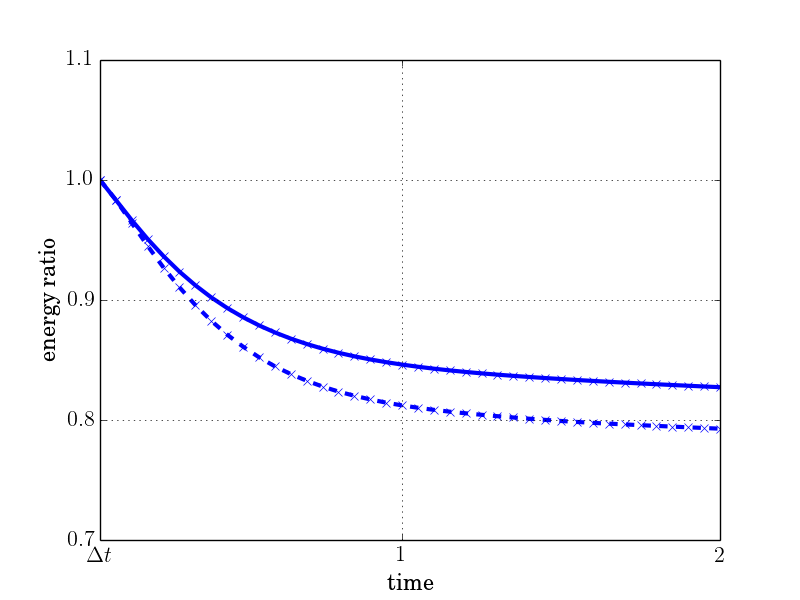}}
 \subcaptionbox{\tiny{$\Delta t = 5\cdot 10^{-2}$, $h_s = 1/32$}.}
   {\includegraphics[scale=.2]{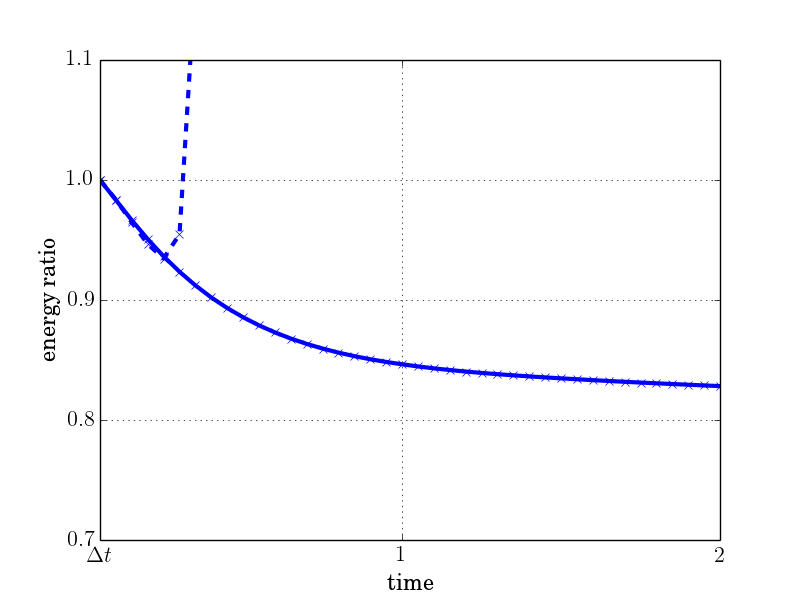}}
 \caption{Energy ratio \NEW{(see Equation~\eqref{eq:ener_comp})} for codimension one structure. The structure elastic constant $\kappa = 5$, $h_x = 1/32$,
 the fluid viscosity $\nu = 1$, $\delta\rho = 0.3$. The solid line correspond to the \ibmdlm scheme, while 
 the dashed line marks the energy for the \feibm scheme.
 }
 \label{fig:thin_energy_rho03}
 \end{figure}
 
 \begin{figure}
 \centering
 \subcaptionbox{\tiny{$\Delta t = 10^{-1}$, $h_x = 1/4$}.}
   {\includegraphics[scale=.2]{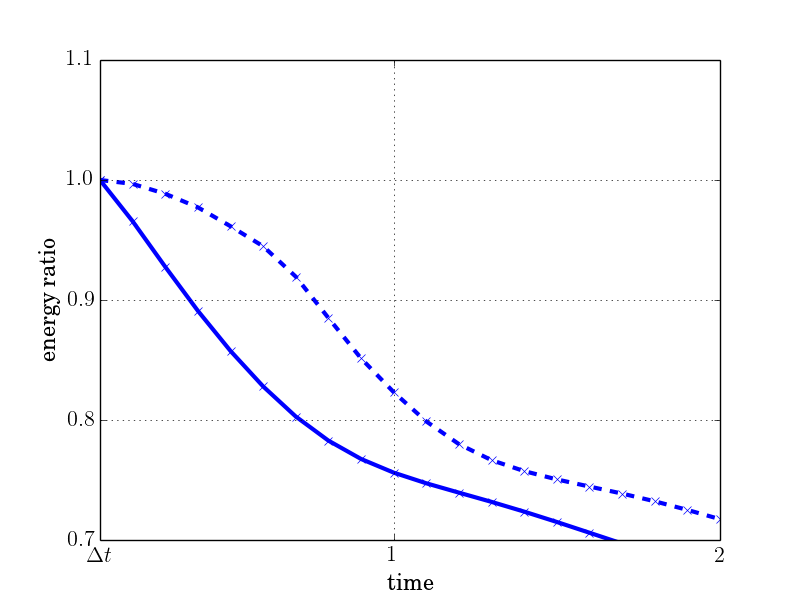}}
 \subcaptionbox{\tiny{$\Delta t = 10^{-1}$, $h_x = 1/8$}.}
   {\includegraphics[scale=.2]{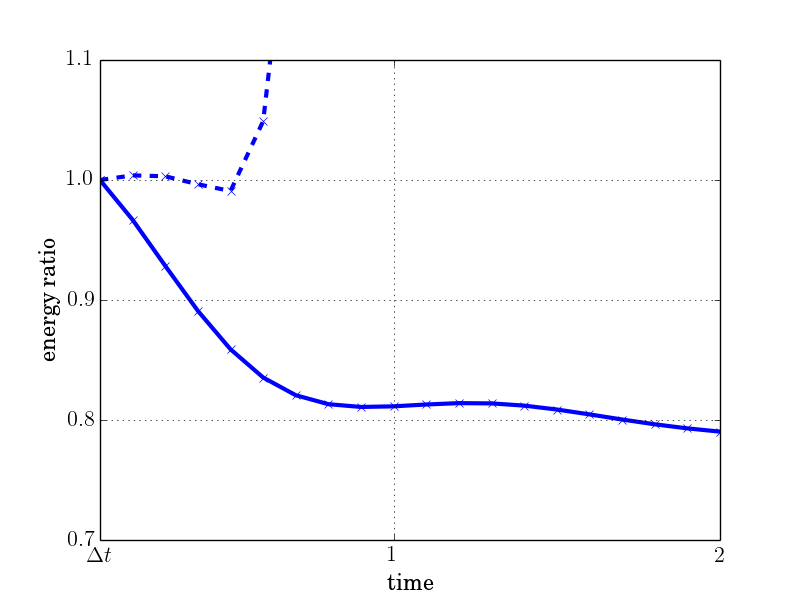}}
 \subcaptionbox{\tiny{$\Delta t = 10^{-1}$, $h_x = 1/16$}.}
   {\includegraphics[scale=.2]{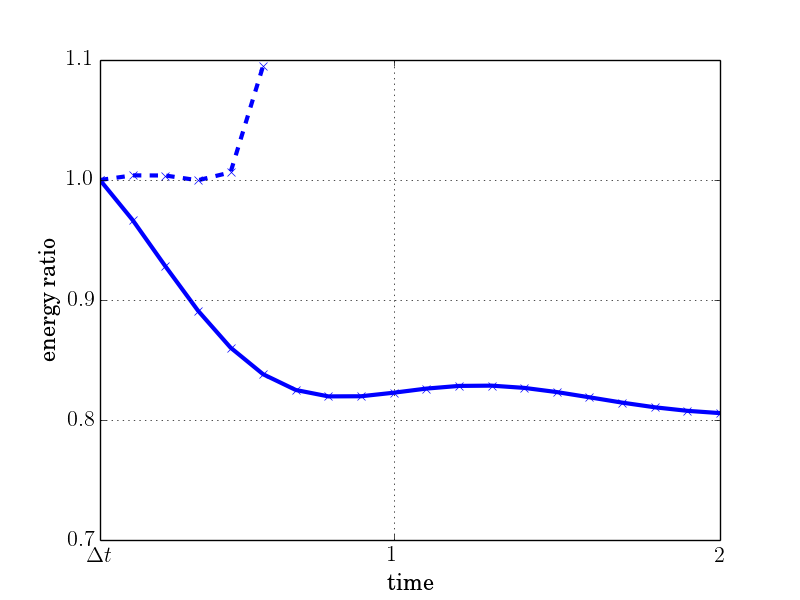}}\\
  \subcaptionbox{\tiny{$\Delta t = 5\cdot10^{-2}$, $h_x = 1/4$}.}
   {\includegraphics[scale=.2]{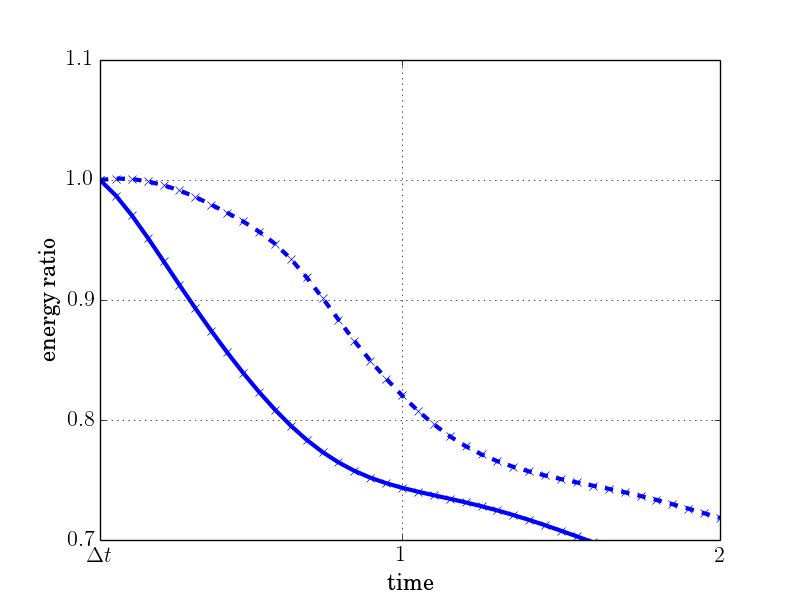}}
 \subcaptionbox{\tiny{$\Delta t = 5\cdot10^{-2}$, $h_x = 1/8$}.}
   {\includegraphics[scale=.2]{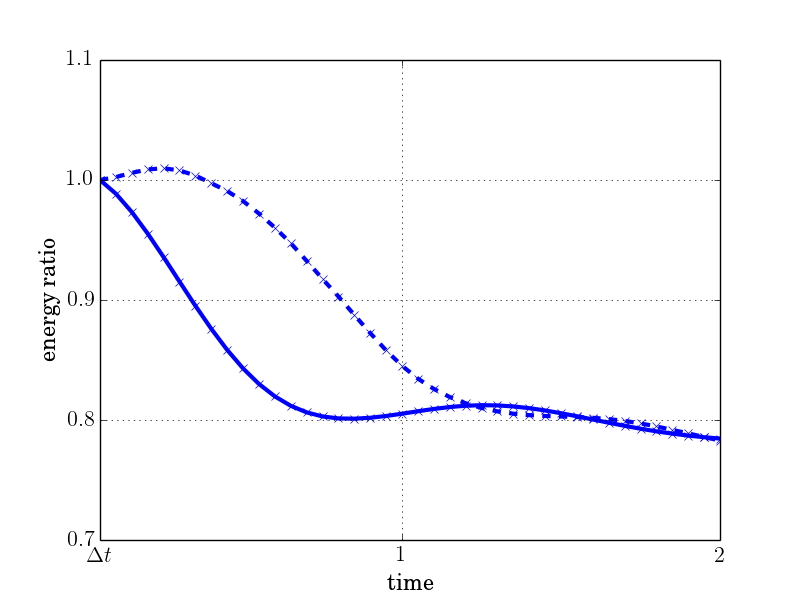}}
 \subcaptionbox{\tiny{$\Delta t = 5\cdot10^{-2}$, $h_x = 1/16$}.}
   {\includegraphics[scale=.2]{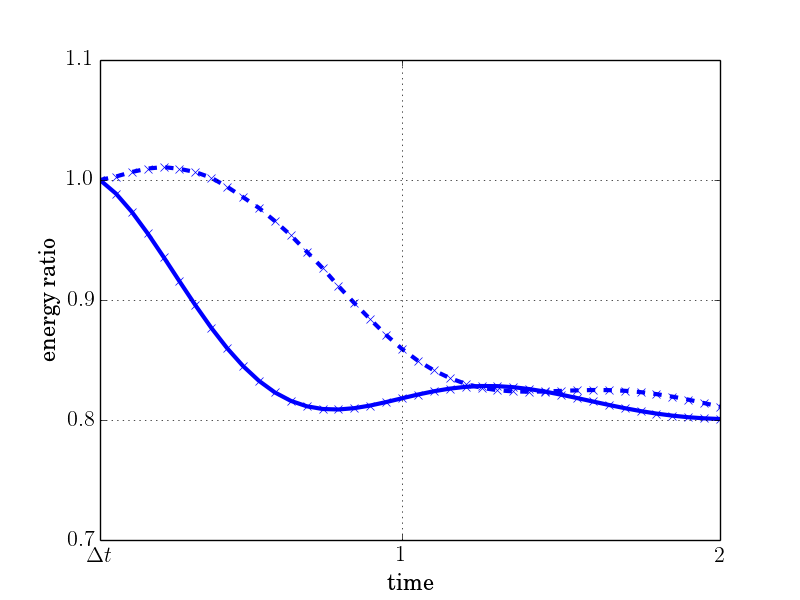}}
 \caption{Energy ratio \NEW{(see Equation~\eqref{eq:ener_comp})} for codimension zero structure. The structure elastic constant $\kappa = 1$, $h_s = 1/8$,
 the fluid viscosity $\nu = 0.05$, $\delta\rho = 0.3$. The solid line correspond to the \ibmdlm scheme, while 
 the dashed line marks the energy for the \feibm scheme.}
 \label{fig:thick_energy_rho03}
 \end{figure}

Since the fluid considered in our experiments is incompressible, an important
physical property that has to be preserved is the mass conservation of the
coupled scheme (see~\cite{bcgg2012,bcggumi} for related work in this
framework). It turns out that, unexpectedly, the \ibmdlm scheme enjoys better
conservation properties than the \feibm. We do not have a theoretical
explanation for this phenomenon yet; in Figure~\ref{fg:masscons} we report the
comparison between the two cases (structure of codimension one). 
In Figures~\ref{subfig:str_dlm} and~\ref{subfig:str_ibm} we report the
evolution of the structure position during the simulation.
Here it is already clear how better is the mass preservation for the \ibmdlm
scheme. To make this result even clearer, in Figure~\ref{subfig:str_cfr} we present the final 
structure position for both schemes.

\begin{figure}
\centering
\subcaptionbox{\ibmdlm.\label{subfig:str_dlm}}
    {\includegraphics[scale=.3]{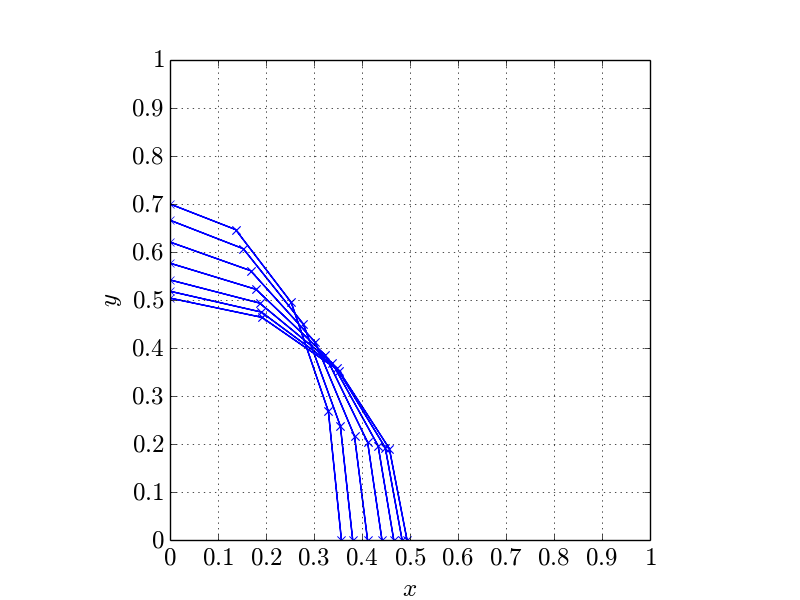}}
\subcaptionbox{\feibm.\label{subfig:str_ibm}}
    {\includegraphics[scale=.3]{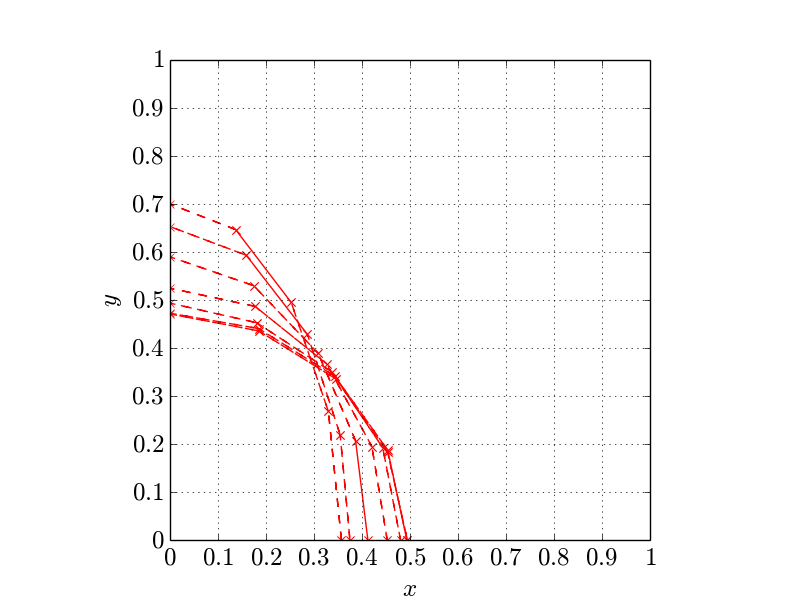}}
\subcaptionbox{Final position for the two schemes.\label{subfig:str_cfr}}
    {\includegraphics[scale=.3]{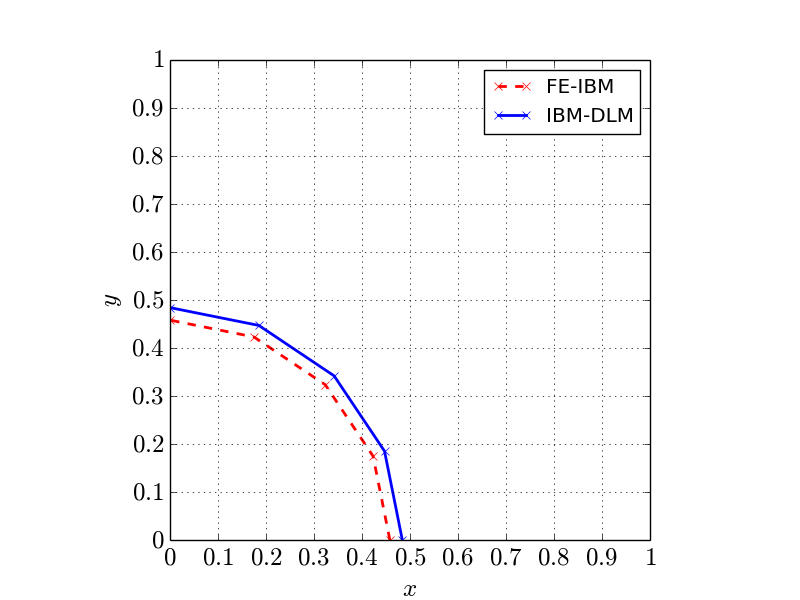}}
\caption{Comparison of mass conservation for the \ibmdlm (left) and \feibm
(right) schemes}
\label{fg:masscons}
\end{figure}

This phenomenon is more evident when higher order schemes are used. For
instance, in Figure~\ref{fg:p3p2} we show the results of the same simulation
when the enhanced Hood--Taylor Stokes element $P_3-(P_2^c+P_1)$ is used. The
justification of this effect is currently under investigation.
\NEW{More numerical results about this aspect (including some quantitative
analysis of the area loss) are reported in~\cite{ecmi}}.

\begin{figure}
\includegraphics[width=6cm]{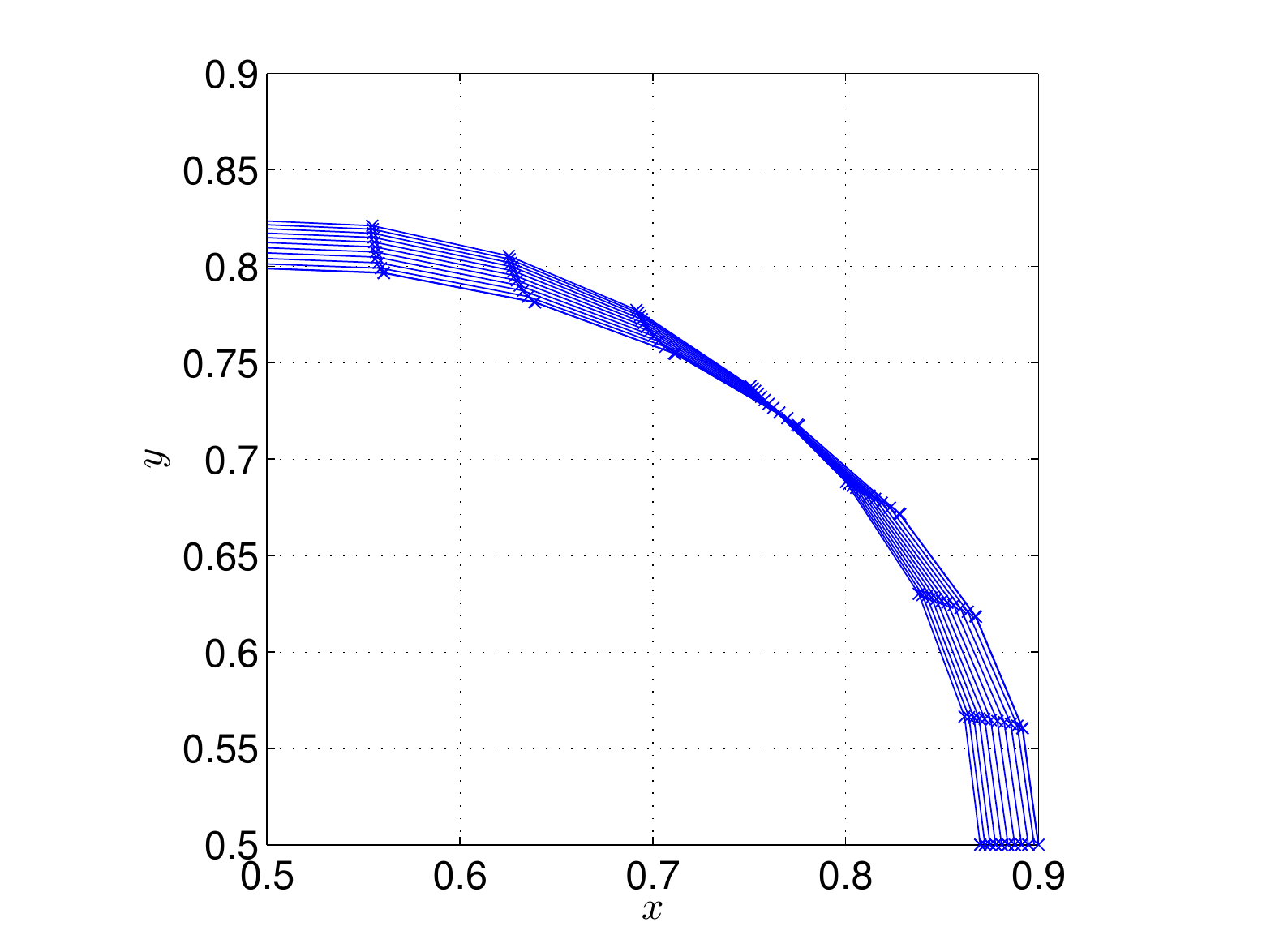}
\includegraphics[width=6cm]{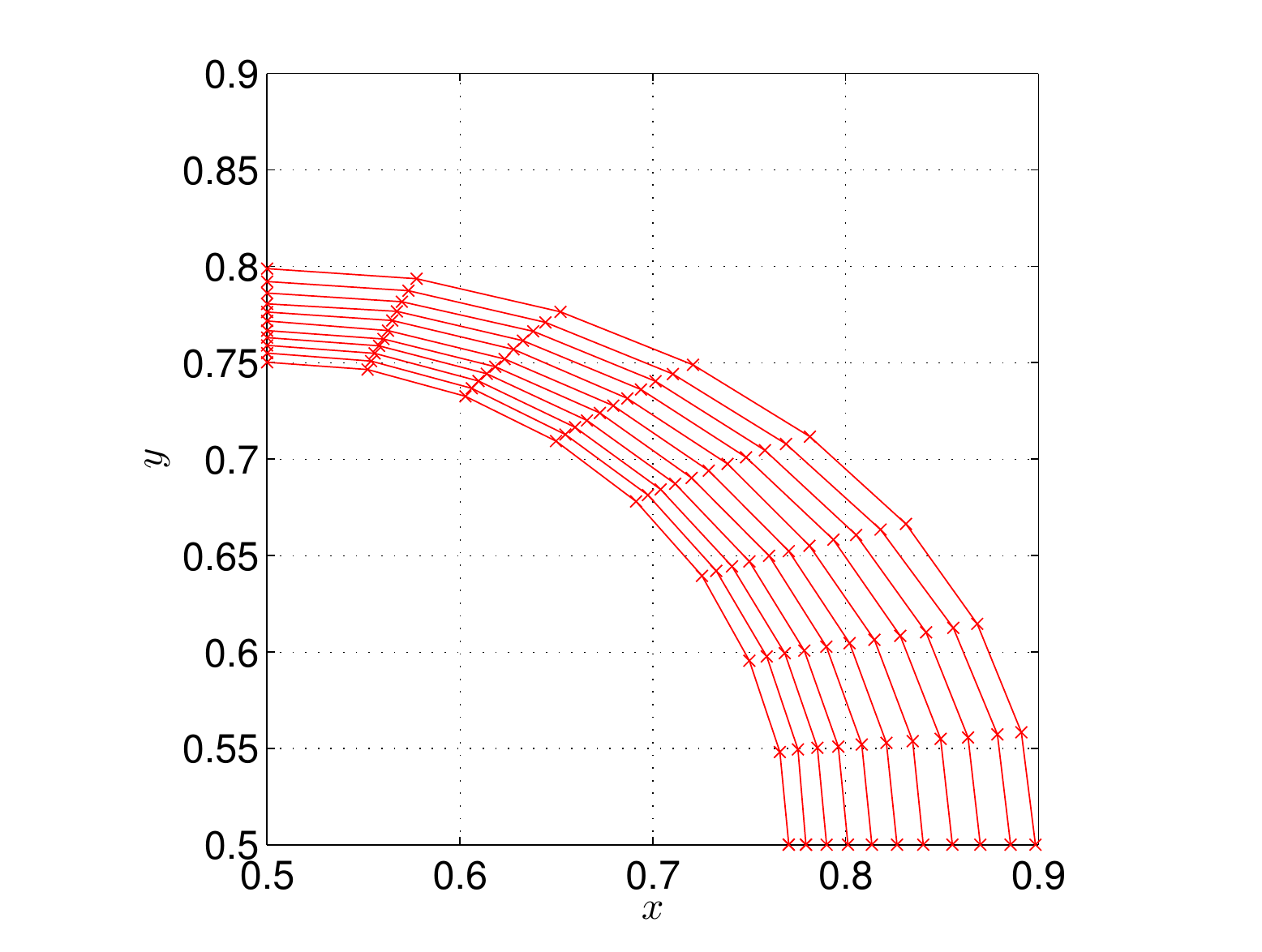}
\caption{Mass conservation of the \ibmdlm (left) and \feibm (right) with
higher order fluid element.\label{fg:p3p2}}
\end{figure}

\NEW{Our last numerical tests aim at checking numerically the convergence
order of our scheme.

We start with the convergence in space for the \ibmdlm applied to a
codimension one setting. We consider a simple (steady) test with 
analytical exact solution. The initial position of the structure is
a circle immersed in a fluid at rest, so that the asymptotic configuration
remains unchanged.
Table~\ref{tab:dlm_accuracy} reports the rate of convergence which is
perfectly in agreement with the regularity of the solution. Being the 
pressure solution discontinuous, the optimal convergence rate is $1.5$ for the velocity
in $L^2$ and $0.5$ for the pressure in $L^2$.

Then, we check the convergence in time. This point is particularly interesting, 
since the kinematic constraint $\mathsf{L}_f$ is enforced using the domain configuration at the
previous time step. Table~\ref{tab:dlm_time_conv} shows that this does not
affect the first order of convergence of the Euler scheme. The test case
in this situation consists of a codimension zero setting, corresponding to the
framework of Figures~\ref{fig:ellpse_vel_thick}
and~\ref{fig:ellpse_prex_thick}. In this case there is no analytical solution. 
A reference solution, denoted with the subscript $\mathit{ex}$, is
obtained with $\Delta t = 10^{-3}$: all errors are measured with respect to this reference
solution.
}

\begin{table}
 \NEW{
\begin{tabular}{cc cc cc cc cc}\toprule
$h_x$ & & $||p-p_h||_{L^2}$ & & $L^2$-rate & & $||\mathbf{u}-\mathbf{u}_h||_{L^2}$ & & $L^2$-rate & \\\cmidrule{1-1}\cmidrule{3-5}\cmidrule{7-9}
$1/4$& & 2.96063&  & - &  &0.02225&  & - &  \\\cmidrule{1-1}\cmidrule{3-5}\cmidrule{7-9}
$1/8$& & 2.10271&  &   0.49&  &0.01022&  &   1.12&  \\\cmidrule{1-1}\cmidrule{3-5}\cmidrule{7-9}
$1/16$& & 1.43488&  &   0.55&  &0.00392&  &   1.38&  \\\cmidrule{1-1}\cmidrule{3-5}\cmidrule{7-9}
$1/24$& & 1.15722&  &   0.53&  &0.00212&  &   1.52&  \\\cmidrule{1-1}\cmidrule{3-5}\cmidrule{7-9}
$1/32$& & 0.97502&  &   0.60&  &0.00134&  &   1.60&  \\\cmidrule{1-1}\cmidrule{3-5}\cmidrule{7-9}
$1/40$& & 0.88740&  &   0.42&  &0.00102&  &   1.22&  \\\cmidrule{1-1}\cmidrule{3-5}\cmidrule{7-9}
$1/64$& & 0.69442&  &   0.52&  &0.00052&  &   1.43&  \\\bottomrule
\end{tabular}
}

 \caption{\NEW{Spatial convergence for the \ibmdlm scheme.
In each case the mesh size for the structure is the same as for the fluid.}}
 \label{tab:dlm_accuracy}
\end{table}

\begin{table}
 \NEW{
\begin{tabular}{cc cc cc cc cc}\toprule
$\Delta t$ & & $||\mathbf{X}_{\mathit{ex}}-\mathbf{X}_h||_{L^2}$ & & $L^2$-rate & & $||\mathbf{u}_{\mathit{ex}}-\mathbf{u}_h||_{L^2}$ & & $L^2$-rate & \\\cmidrule{1-1}\cmidrule{3-5}\cmidrule{7-9}
1$\cdot 10^{-1}$& & 5.54945e-06&  & - &  &1.65152e-05&  & - &  \\\cmidrule{1-1}\cmidrule{3-5}\cmidrule{7-9}
5$\cdot 10^{-2}$& & 2.73334e-06&  &   1.02&  &7.92803e-06&  &   1.06&  \\\cmidrule{1-1}\cmidrule{3-5}\cmidrule{7-9}
2$\cdot 10^{-2}$& & 1.05724e-06&  &   1.04&  &3.01373e-06&  &   1.06&  \\\cmidrule{1-1}\cmidrule{3-5}\cmidrule{7-9}
1$\cdot 10^{-2}$& & 5.00445e-07&  &   1.08&  &1.41808e-06&  &   1.09&  \\\bottomrule
\end{tabular}
}

 \caption{\NEW{Time convergence for \ibmdlm test case.
 The mesh sizes for the fluid and for the structure are $1/16$.}}
 \label{tab:dlm_time_conv}
\end{table}

We conclude this section with some algorithmic comments that might be
useful in order to enhance the solution accuracy.
The solution of system~\eqref{eq:matriciona} can be interpreted as a
semi-implicit discretization of the original fluid-structure problem. The
strategy we are going to describe represents a first investigation towards
the approximation of a fully implicit scheme.

The resulting algorithm (based on a fixed point iteration) is pretty simple:
in the following the discretization index $h$ will be omitted for the sake of
clearness.
Consider the $n$-th time step and the $k$-th iteration and denote
$\X_0 = \X^n$. For $k\ge1$ we solve:
\[
\left(
\begin{array}{cc|c|c}
\mathsf{A}&\mathsf{B}^\top&0&\mathsf{L}_f(\X_{k-1})^\top\\
\mathsf{B}&0&0&0\\
\hline\\[-10pt]
0&0&\mathsf{A}_s&-\mathsf{L}^\top_s\\
\hline\\[-10pt]
\mathsf{L}_f(\X_{k-1})&0&-\mathsf{L}_s&0
\end{array}
\right)
\left(
\begin{array}{c}
\u_{k}\\p_{k}\\
\hline\\[-10pt]
\X_{k} \\
\hline\\[-10pt]
\llambda_k
\end{array}
\right)=
\left(
\begin{array}{c}
\mathsf{f}\\0\\
\hline\\[-10pt]
\mathsf{g}\\
\hline\\[-10pt]
\mathsf{d}
\end{array}
\right).
\]
We consider the fluid-structure coupling residual as
\begin{equation}
r\left(\u_k,\X_k\right) = \left\|\mathsf{L}_f(\X_{k})\u_{k}-\frac{\mathsf{L}_s \X_{k}-\mathsf{L}_s \X^n}{\Delta t}\right\|_{0,\B}.
\label{eq:fsi_res}
\end{equation}
If $r\left(\u_k,\X_k\right)\le\varepsilon$ we set $(\u^{n+1},p^{n+1},\llambda^{n+1},\X^{n+1})= (\u_k,p_k,\llambda_k,\X_k)$,
otherwise continue iterating.
The residual values for ten iterations, (corresponding to ten consecutive time
steps) are represented in Figure~\ref{fig:nnl_iter}.
Both codimension one and zero are considered.
The simulation parameters are specified in the figure caption.
Different lines correspond to different time steps (generally, the top lines
correspond to the first time step, and so on monotonically down to the bottom
line corresponding to the tenth time step).
It is interesting to notice that the residual 
decreases as the structure approaches a stationary state. Moreover, and most
importantly, within the same time step,
the algorithm converges with approximately first order. This property is
important in view of designing multigrid type algorithms in order to reduce
the computational cost of the implicit scheme. \NEW{The natural evolution of this 
approach would combine fixed point iterations with a multigrid V-cycle. A
crucial aspect is to reduce the computational cost needed to assembling the term
$\mathsf{L}_f(\X_{k-1})$. Future works will be addressed to tackle this issue.}

\begin{figure}
 \centering
 \subcaptionbox{Codimension one.}
   {\includegraphics[scale=.4]{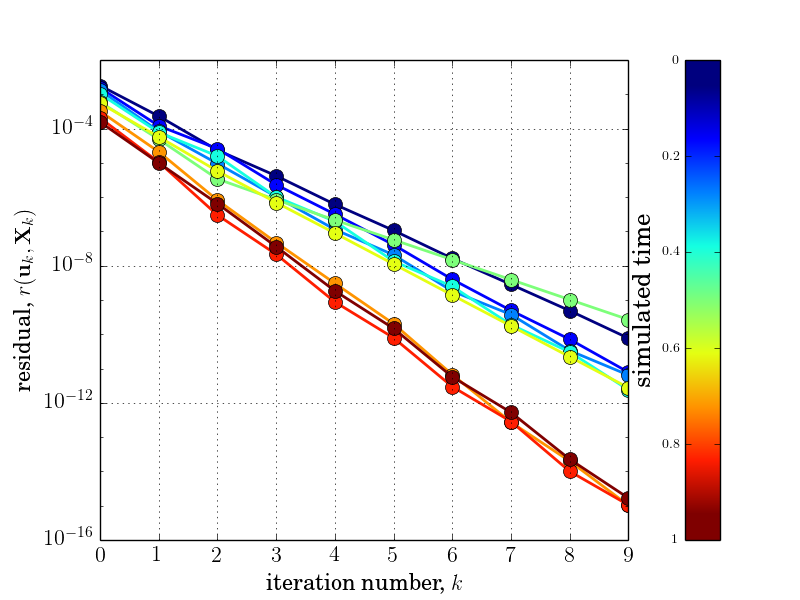}}
 \subcaptionbox{Codimension zero.}
   {\includegraphics[scale=.4]{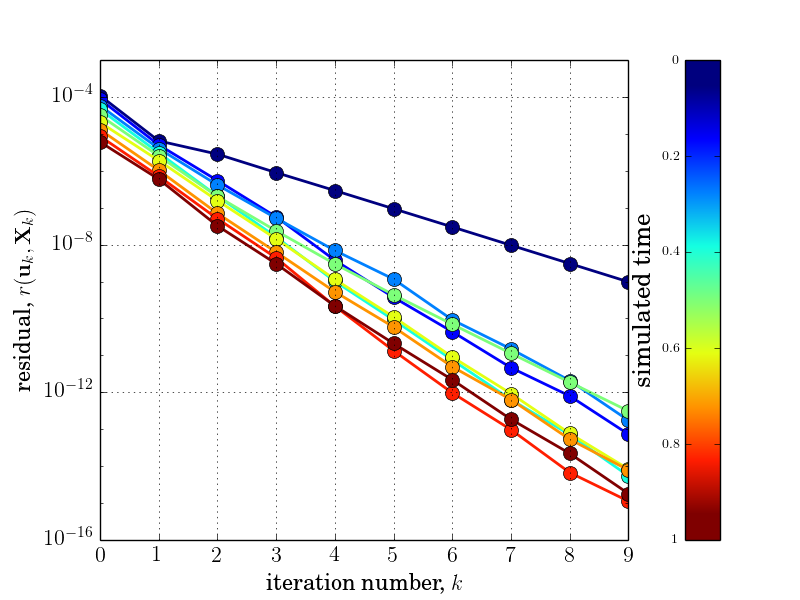}}
 \caption{Convergence for fixed point iteration for the fluid-structure
interaction 
 nonlinear coupling. The simulation parameters are $h_x=h_s = 1/8$, $\kappa = 8$, 
 $\Delta t = 1/10$, results are presented for codimension one and zero. The residual 
 is defined in equation~\eqref{eq:fsi_res}. Different color refer to different
 simulated times.}
 \label{fig:nnl_iter}
 \end{figure}
 
Finally, in Figure~\ref{fig:square} we present some snapshots of the
simulation of a codimension zero solid which is a square at its equilibrium
configuration.
In the initial configuration the solid is stretched and has a rectangular
shape.
At the beginning of the simulation, Figure~\ref{fig:sq00}, the corner of the
structure imposes a singularity to velocity field and a vortex arises.
In snapshot~\ref{fig:sq01} the structure is bouncing along the $x$ axis.
In Figure~\ref{fig:sq02} the structure interacts again with the fluid, as it
is heading to the equilibrium position, a second vortex arises as a
consequence of the structure motion. In Figure~\ref{fig:sq03} the structure
has almost approached its equilibrium position.
 
\begin{figure}
 \centering
 \subcaptionbox{$t = 0.03$.\label{fig:sq00}}
   {\includegraphics[scale=.3]{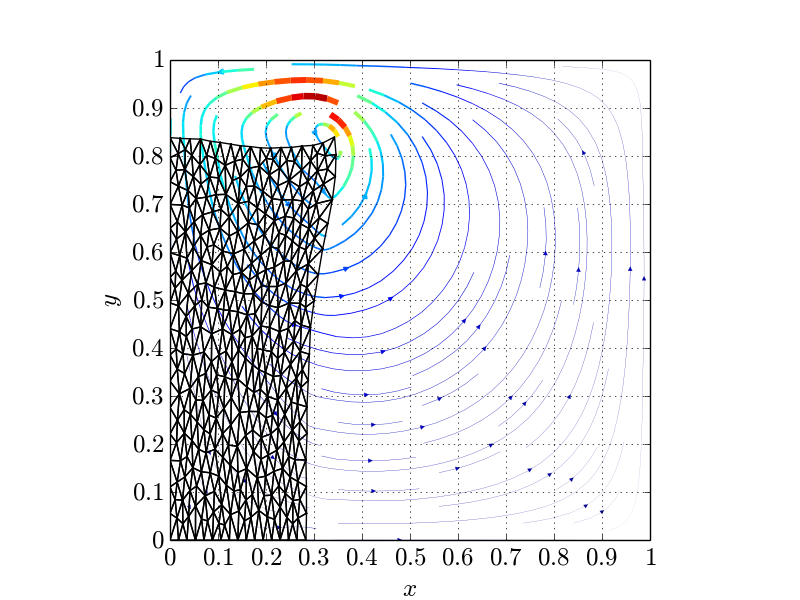}}
 \subcaptionbox{$t = 0.17$.\label{fig:sq01}}
   {\includegraphics[scale=.3]{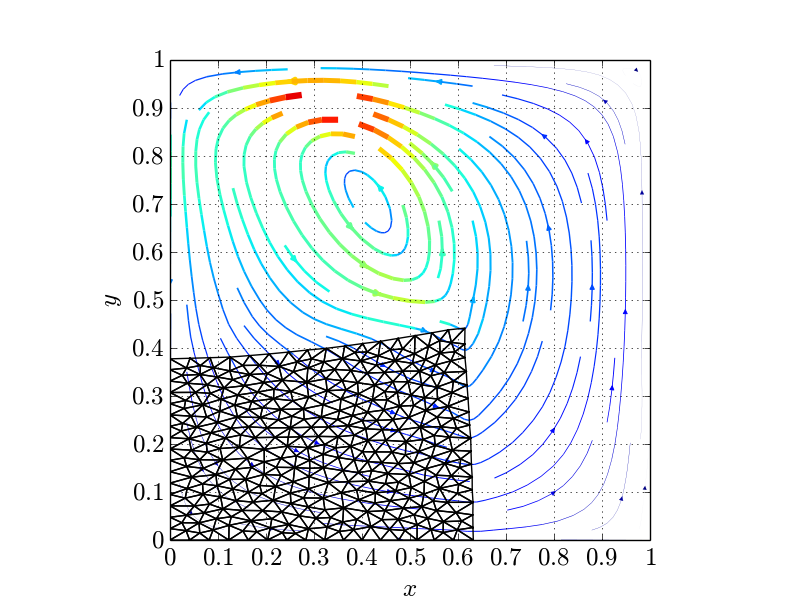}}
 \subcaptionbox{$t = 0.24$.\label{fig:sq02}}
   {\includegraphics[scale=.3]{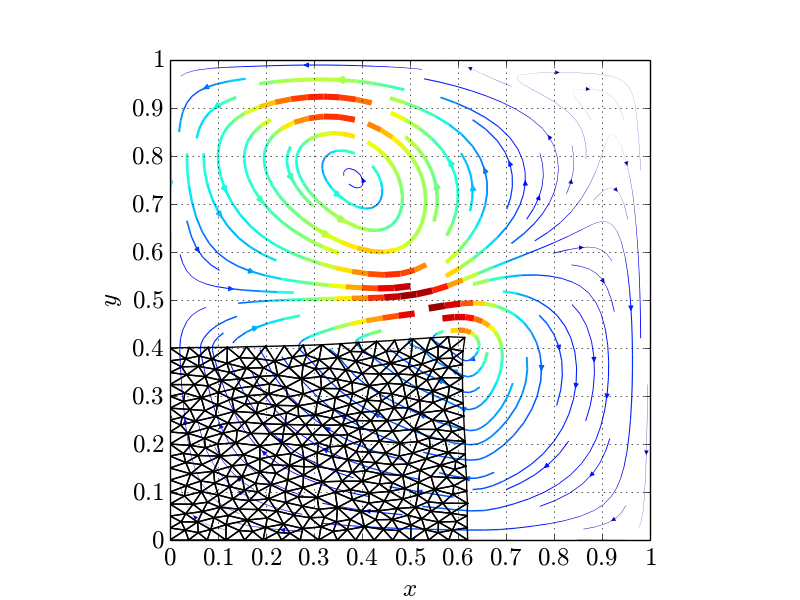}}
 \subcaptionbox{$t =2$.\label{fig:sq03}}
   {\includegraphics[scale=.3]{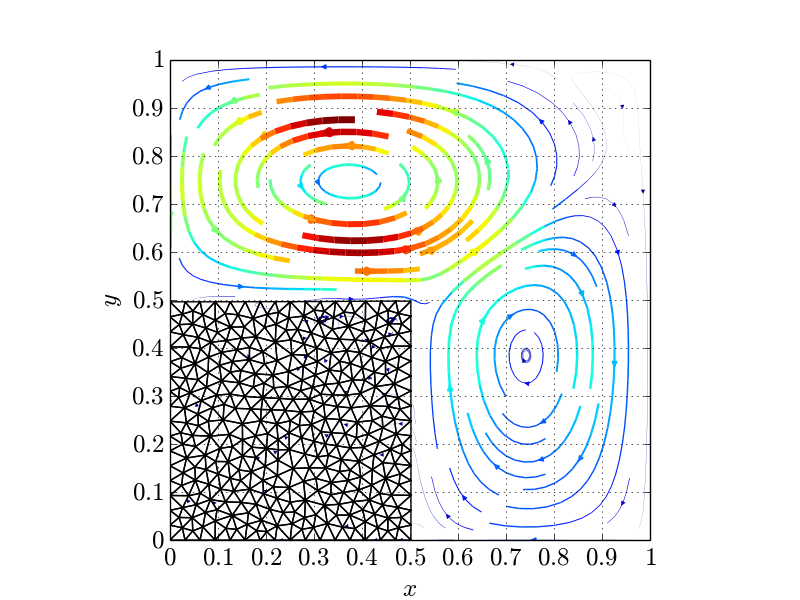}}
 \caption{Simulation of a square structure that moves from an initial rectangular position to 
 the equilibrium position. The fluid mesh $h_x = 1/32$, the structure one is $h_s = 1/16$, $\Delta t = 10^{-3}$.
 The fluid viscosity is $\nu = 0.01$ and the structure elastic constant is $100$.}
\label{fig:square}
\end{figure}

\section{Conclusions}
In this paper we presented a new formulation for the finite element
approximation of fluid-structure interaction problems within the setting of
the Immersed Boundary Method. With this formulation the coupling between the
fluid and the structure is modeled with the help of a distributed Lagrange
multiplier, so that a fully variational problem is obtained. The main feature
of the fully discrete scheme associated with this formulation, is that its
stability does not require any restriction on the time step size.

Numerical experiments confirm the theoretical energy estimates.

\bibliographystyle{plain}
\bibliography{ref}

\end{document}